\documentclass[draft]{jaums-arxiv}

\usepackage[mathscr]{eucal}
\usepackage{amssymb,bm,bbm} 
\usepackage{enumitem,graphicx}
\usepackage{verbatim,url}

\usepackage{amssymb}
\usepackage{amsmath}
\usepackage{amsthm}
\usepackage{graphicx}
\usepackage[mathscr]{eucal}
\usepackage{verbatim,url}
\usepackage{xy}
\usepackage{calc}
\usepackage{xcolor,pgf,tikz}

\numberwithin{equation}{section}

\theoremstyle{plain}

\newtheorem{thm}{Theorem}[section]
\newtheorem{lem}[thm]{Lemma}
\newtheorem{prop}[thm]{Proposition}

\newtheorem{claim}{Claim}

\theoremstyle{definition}
\newtheorem{df}[thm]{Definition}

\newtheorem{prob}[thm]{Problem}

\definecolor{myred}{cmyk}{0.2,1,1,0.3}
\definecolor{mygreen}{cmyk}{1,0.4,1,0.3}
\definecolor{myblue}{cmyk}{1,0.5,0.4,0.3}
\definecolor{mypurple}{cmyk}{0.5,1,0.5,0.3}
\definecolor{mygrey}{cmyk}{0,0,0,0.8}

\usetikzlibrary{calc,positioning,shapes,arrows,arrows.meta}
\tikzset{%
 shaded/.style={draw, shape=circle, fill=black!35, inner sep=1.4pt},
 unshaded/.style={draw, shape=circle, fill=white, inner sep=1.4pt},
 order/.style={thin},
 label/.style={shape=rectangle, inner sep=6pt},
 auto,
 map/.style={->, densely dashed, shorten >=5pt, shorten <=5pt, >=stealth', looseness=1.1},
 curvy/.style={->, densely dashed, shorten >=5pt, shorten <=12pt, >=stealth', looseness=1.1},
 operationgj/.style={->, densely dashed, shorten >=5pt, shorten <=20pt, >=stealth', looseness=1.1},
 relationlejk/.style={->, shorten >=5pt, shorten <=5pt, >=stealth'},
 }

\newcommand{\eusbA}{\medsub e {\kern-0.75pt\A\kern-0.75pt}}

\newcommand{\cat}[1]{\boldsymbol{\mathscr{#1}}}

\newcommand{\CCD}{\cat D}

\newcommand{\CP}{\cat P}
\newcommand{\CM}{\cat M}
\newcommand{\CV}{\cat V}
\newcommand{\CX}{\cat X}
\newcommand{\CY}{\cat Y}

\font\bmi=cmmi8 scaled 1440
\newcommand{\powerset}{\raise.6ex\hbox{\bmi\char'175 }}

\newcommand{\twiddle}[1]{\mathbb{#1}} 

\newcommand{\MT}{\twiddle M}

\newcommand{\Tp}{\mathscr{T}}

\newcommand{\A}{\mathbf A}
\newcommand{\B}{\mathbf B}
\newcommand{\C}{\mathbf C}

\newcommand{\K}{\mathbf K}
\newcommand{\F}{\mathbf F}
\newcommand{\M}{\mathbf M}

\renewcommand{\P}{\mathbf P}
\newcommand{\Q}{\mathbf Q}

\newcommand{\T}{\mathbf T}
\newcommand{\X}{\mathbb X}
\newcommand{\Y}{\mathbb Y}

\newcommand{\Ys}{\Y\kern -2pt _s}

\renewcommand{\t}{\mathrm t} 
\renewcommand{\k}{\mathrm k} 

 \DeclareMathOperator{\graph}{graph}
 
 \DeclareMathOperator{\id}{id}
 \DeclareMathOperator{\Var}{Var}
 \DeclareMathOperator{\Sub}{Sub} 
 \DeclareMathOperator{\rank}{rnk} 
 
 \DeclareMathOperator{\ISP}{\mathsf{ISP}}
 \DeclareMathOperator{\HSP}{\mathsf{HSP}}
 \DeclareMathOperator{\IScP}{{\mathsf{IS} _{\mathrm{c}}
 \mathsf{P}^+}}

\newcommand{\comp}{{\setminus}}
\newcommand{\rest}[1]{{\upharpoonright}_{#1}}
\newcommand{\bigand}{\mathop{\bigwedge\kern -8.5truept \bigwedge}}
\newcommand{\Bigand}{\mathop{\bigwedge\kern -10truept \bigwedge}}
\newcommand{\littleand}{\mathbin{\wedge\kern -8truept \wedge}}
\newcommand{\bigor}{\mathop{\bigvee\kern -8.5truept \bigvee}}
\newcommand{\Bigor}{\mathop{\bigvee\kern -10truept \bigvee}}
\newcommand{\littleor}{\mathbin{\vee\kern -8truept \vee}}
\renewcommand{\le}{\leqslant}
\newcommand{\nle}{\nleqslant}
\renewcommand{\ge}{\geqslant}

\newcommand{\lez}{\leqslant^0}
\newcommand{\gez}{\geqslant^0}
\newcommand{\lej}{\leqslant^j}

\newcommand{\lek}{\leqslant^k}
\newcommand{\lel}{\leqslant^\ell}
\newcommand{\gek}{\geqslant^k}
\newcommand{\leP}{\preccurlyeq} 
 
\newcommand{\geP}{\succcurlyeq} 
\newcommand{\lePY}{\le} 
\newcommand{\nlePY}{\not\le} 
 
\newcommand{\lejl}{\leqslant^{j\ell}}
\newcommand{\leji}{\leqslant^{ji}}
\newcommand{\lekl}{\leqslant^{k\ell}}
\newcommand{\leil}{\leqslant^{i\ell}}
\newcommand{\lejk}{\leqslant^{jk}}

\newcommand{\lekk}{\leqslant^{kk}}
\newcommand{\gekj}{\geqslant^{kj}}

\newcommand{\nlejk}{\nleqslant^{jk}}

\newcommand{\nlez}{\nleqslant^0}

\newcommand{\nlek}{\nleqslant^{k}}

\newcommand{\du}{\mathbin{\dot\cup}}

\newcommand{\up}{{\uparrow}}
\newcommand{\down}{{\downarrow}}
\newcommand{\lsem}{[\kern-1.75pt[}
\newcommand{\rsem}{]\kern-1.75pt]}
\newcommand{\conv}{{}^{\kern2pt\raise-3pt\hbox{$\breve{}$}}} 
\newcommand{\convsub}[1]{\rlap{$^{\kern2pt\raise-3pt\hbox{$\breve{}$}}$}{}_{#1}} 
\newcommand{\Rsub}[1]{\mathcal R_{#1}}

\newcommand{\FOUR}{\mathcal{FOUR}}

\newcommand{\JB}{\mathbf{J}}

\newcommand{\mbf}{\boldsymbol{f}}
\newcommand{\mbt}{\boldsymbol{t}}

\newcommand{\bsF}{F} 
\newcommand{\bsT}{T} 
\newcommand{\bF}{F_{\kern -1pt n}} 
\newcommand{\bT}{T_{\kern -1pt n}} 
\newcommand{\bFA}{\mathbf F_{\kern -1pt \A}} 
\newcommand{\bFB}{\mathbf F_{\kern -1pt \B}} 
\newcommand{\bTB}{\mathbf T_{\kern -1pt \B}} 
\newcommand{\bFC}{\mathbf F_{\kern -1pt \C}} 
\newcommand{\bTC}{\mathbf T_{\kern -1pt \C}} 
\newcommand{\bFn}{\mathbf F_{\kern -1pt n}} 
\newcommand{\bTn}{\mathbf T_{\kern -1pt n}} 

\newcommand{\two}{\boldsymbol 2}
\newcommand{\one}{\boldsymbol 1}
\newcommand{\zero}{\boldsymbol 0}

\begin{document}


\title[Expanding Belnap 2]{Expanding Belnap 2: the dual category\\ in depth}

\author[A. P. K. Craig]{Andrew P. K. Craig}
\address{Department of Mathematics\\
and Applied Mathematics\\
University of Johannesburg\\PO Box 524, Auckland Park, 2006\\South~Africa\\
and\\
Department of Mathematics and Statistics\\La Trobe University\\Victoria 3086\\Australia
}
\email{acraig@uj.ac.za}

\author[B. A. Davey]{Brian A. Davey}
\address{Department of Mathematics and Statistics\\La Trobe University\\Victoria 3086\\Australia}
\email{b.davey@latrobe.edu.au}

\author[M. Haviar]{Miroslav Haviar}
\address{Department of Mathematics\\Faculty of Natural Sciences,\\ 
M. Bel University\\Tajovsk\'eho 
40, 974~01 Bansk\'a Bystrica\\Slovakia\\
and\\
Department of Mathematics\\
and Applied Mathematics\\
University of Johannesburg\\PO Box 524, Auckland Park, 2006\\South~Africa
}
\email{miroslav.haviar@umb.sk}

\subjclass{06D50, 08C20, 03G25}

\keywords{bilattice, default bilattice, natural duality, multi-sorted natural duality, Priestley duality, piggyback duality}

\begin{abstract}
Bilattices, which provide an algebraic tool for simultaneously
modelling knowledge and truth, were introduced by N.\,D.~Belnap in a
1977 paper entitled \emph{How a computer should think}.
Prioritised default bilattices include not only Belnap's four values,
 for `true' ($\mbt$),
`false'($\mbf$), `contradiction'  ($\top$) and `no information' ($\bot$),
 but also indexed families of default values for simultaneously modelling degrees of knowledge and truth.
Prioritised default bilattices have applications in a number of areas including artificial intelligence.

In our companion paper, we introduced a new family of
prioritised default bilattices, $\mathbf J_n$, for $n \in \omega$,
with $\mathbf J_0$ being Belnap's seminal example. We gave
a duality for the variety $\CV_n$ generated by $\JB_n$, with the objects of the dual category $\CX_n$ being multi-sorted topological structures.

Here we study the dual category in depth. We give an axiomatisation of the category $\CX_n$ and show that it is isomorphic to a category $\CY_n$ of single-sorted topological structures. The objects of $\CY_n$ are Priestley spaces endowed with a continuous retraction in which the order has a natural ranking. We show how to construct the Priestley dual of the underlying bounded distributive
lattice of an algebra in $\CV_n$ via its dual in $\CY_n$; as an application we show that the size of the free algebra $\F_{\CV_n}(1)$ is given by a polynomial in $n$ of degree~$6$.
\end{abstract}

\thanks{The first author acknowledges the support of the National Research Foundation of South Africa (grant 127266) and the third author acknowledges the support of Slovak grant VEGA 1/0337/16.}

\date{\emph{Dedicated to the memory of Moshe S. Goldberg}}

\vspace{-2cm}

\maketitle


\vspace{-0.9cm}

\section{Introduction}\label{sec:intro}

In our companion paper~\cite{CDH19}, we introduced a new class
$\{\, \JB_n \mid n \in \omega\,\}$ of default bilattices
for use in prioritised default logic. The first of these
bilattices,~$\JB_0$, is Belnap's famous four-element bilattice known
as $\FOUR$~\cite{Bel-think}, while for $n \geqslant 1$, the
bilattice $\JB_n$ provides a new algebraic structure for dealing
with inconsistent and incomplete information.
The importance of our prioritised default bilattices in comparison
with those previously studied is that in our family there is no
distinction between the level at which the contradictions or
agreements take place. Any contradictory response that includes some
level of truth ($\mbt_i$) and some level of falsity ($\mbf_j$) is
registered as a total contradiction ($\top$) and a total lack of
consensus ($\bot$). This can lead to improvements in existing
applications of default bilattices. 

Default bilattices, and more generally prioritised default bilattices, have a rich history. We will mention just a few examples.
A logic for default reasoning was introduced by Reiter~\cite{R80}.
The distinction between definite consequences and default
consequences, as well as the notion of inference using bilattices,
were discussed by Ginsberg~\cite{Gins86, Gins88}. He also considered
hierarchies of defaults, pointing out that there is no reason to
assume that the `levels' of default information are discrete.
Prioritised default bilattices now have many applications in
artificial intelligence. Sakama~\cite{S05} studied default theories
based on a 10-valued bilattice and applications to inductive logic
programming. Shet, Harwood and Davis~\cite{Shet2006} proposed a
prioritised multi-valued default logic for identity maintenance in
visual surveillance. Encheva and Tumin~\cite{ET07} applied default
logic based on a 10-element default bilattice in an intelligent
tutoring system as a way of resolving problems with contradictory or
incomplete input.

Bilattices are algebras $\A = \langle A; \otimes, \oplus,
\wedge,\vee, \neg  \rangle$  with two lattice structures, a
knowledge lattice $\A_\k = \langle A; \otimes, \oplus \rangle$, with
associated \emph{knowledge order}~$\le_\k$, and a truth lattice
$\A_\t = \langle A; \wedge,\vee \rangle$, with associated
\emph{truth order}~$\le_\t$, along with an involutive
negation~$\neg$ which is an order automorphism of $\A_\k$ and a dual
order automorphism of~$\A_\t$. Our algebra $\JB_n$ is a \emph{prioritised
default bilattice} as it is equipped with two hierarchies of nullary
operations, $\mbt_i$ and $\mbf_i$, that represent, respectively,
true and false by default. We refer the reader to~\cite{CDH19} for
motivation and background on bilattices in general and prioritised
default bilattices in particular.

In our approach we address mathematical rather than logical aspects
of our prioritised default bilattices. The lack of the much-used
product representation in our context led us to develop a concrete
representation via the theory of natural dualities.
In~\cite{CDH19},
we presented a natural duality between the variety $\CV_n$ generated
by $\JB_n$ and a category $\CX_n$ of multi-sorted topological
structures; see Theorem~\ref{cor:bigmultiduality} below. Our aim
in the present paper is to flesh out the dual category $\CX_n$. 
We begin by giving an axiomatisation of the multi-sorted category~$\CX_n$
(Theorem~\ref{thm:dualcategory}) and then describe an isomorphic
category $\CY_n$ of single-sorted topological structures
(Definition~\ref{df:CYn} and Theorem~\ref{thm:catequiv}). The
objects of $\CY_n$ are Priestley spaces endowed with a continuous
retraction in which the order has a natural ranking. In
Section~\ref{sec:Prdual} we describe the Priestley dual
$\mathrm H(\A^\flat)$ of the underlying bounded distributive
lattice $\A^\flat$ of an algebra $\A$ in $\CV_n$ (Theorem~\ref{thm:trans}). As an application of Theorem~\ref{thm:trans} we show that the size of the free algebra $\F_{\CV_n}(1)$ is given by a polynomial in $n$ of degree~$6$ (Theorem~\ref{thm:sizefree}). Section~\ref{sec:piggy} is devoted to the proof of Theorem~\ref{thm:trans}.

\section{The algebra $\JB_n$ and the variety it generates.}\label{sec:prelim}
In this section we introduce the algebras $\JB_n$, for $n\in \omega$, and recall from~\cite{CDH19} some properties of the variety $\CV_n =\Var(\JB_n)$
generated by~$\JB_n$. Since such intervals occur very frequently in our work, we will use interval notation restricted to the integers; thus we define
\[
[m, n] := \{k\in \mathbb Z \mid m \le k \le n\},
\]
for $m, n\in \mathbb Z$ with $m\le n$.

\begin{df} 
For each $n\in \omega$, let $J_n = \{ \top, \mbf_0, \dots, \mbf_n,\mbt_0, \dots, \mbt_n, \bot \}$. Define the knowledge order, $\le_\k$, and truth order, $\le_\t$, on $J_n$ as in Figure~\ref{fig:Jn}.

\begin{figure}[th]
\centering
\begin{tikzpicture}[scale=0.5]
\begin{scope}[xshift=3.5cm]
  \node[label,anchor=base] at (0,-2.5) {$\JB_n$};
\end{scope}
\begin{scope}
  \node[unshaded] (bot) at (0,0) {};
  \node[unshaded] (fn) at (-1,1) {};
  \node[unshaded] (tn) at (1,1) {};
  \node at (-1,2.2) {$\vdots$};
  \node at (1,2.2) {$\vdots$};
  \node[unshaded] (f1) at (-1,3) {};
  \node[unshaded] (t1) at (1,3) {};
  \node[unshaded] (f0) at (-1,4) {};
  \node[unshaded] (t0) at (1,4) {};
  \node[unshaded] (top) at (0,5) {};
  \draw[order] (bot) -- (fn) -- ($(fn)+(0,0.5)$);
  \draw[order] (bot) -- (tn) -- ($(tn)+(0,0.5)$);
  \draw[order] ($(f1)+(0,-0.5)$) -- (f1) -- (f0) -- (top);
  \draw[order] ($(t1)+(0,-0.5)$) -- (t1) -- (t0) -- (top);
  \node[label,anchor=west,yshift=-3pt] at (bot) {$\bot$};
  \node[label,anchor=east] at (fn) {$\mbf_n$};
  \node[label,anchor=west] at (tn) {$\mbt_n$};
  \node[label,anchor=east] at (f1) {$\mbf_1$};
  \node[label,anchor=west] at (t1) {$\mbt_1$};
  \node[label,anchor=east] at (f0) {$\mbf_0$};
  \node[label,anchor=west] at (t0) {$\mbt_0$};
  \node[label,anchor=west,yshift=3pt] at (top) {$\top$};
  \node[label,anchor=base] at (0,-1.5) {$\le_\k$};
\end{scope}
\begin{scope}[xshift=7cm]
  \node[unshaded] (f0) at (0,0) {};
  \node[unshaded] (f1) at (0,1) {};
  \node at (0,2.2) {$\vdots$};
  \node[unshaded] (fn) at (0,3) {};
  \node[unshaded] (top) at (-1,4) {};
  \node[unshaded] (bot) at (1,4) {};
  \node[unshaded] (tn) at (0,5) {};
  \node at (0,6.2) {$\vdots$};
  \node[unshaded] (t1) at (0,7) {};
  \node[unshaded] (t0) at (0,8) {};
  \draw[order] (f0) -- (f1) -- ($(f1)+(0,0.5)$);
  \draw[order] ($(fn)+(0,-0.5)$) -- (fn) -- (top) -- (tn) -- ($(tn)+(0,0.5)$);
  \draw[order] (fn) -- (bot) -- (tn);
  \draw[order] ($(t1)+(0,-0.5)$) -- (t1) -- (t0);
  \node[label,anchor=west] at (bot) {$\bot$};
  \node[label,anchor=west] at (fn) {$\mbf_n$};
  \node[label,anchor=west] at (tn) {$\mbt_n$};
  \node[label,anchor=west] at (f1) {$\mbf_1$};
  \node[label,anchor=west] at (t1) {$\mbt_1$};
  \node[label,anchor=west] at (f0) {$\mbf_0$};
  \node[label,anchor=west] at (t0) {$\mbt_0$};
  \node[label,anchor=east] at (top) {$\top$};
  \node[label,anchor=base] at (0,-1.5) {$\le_\t$};
  \end{scope}
\end{tikzpicture}

\caption{The knowledge order ($\le_\k$) and truth order ($\le_\t$) on the bilattice $\JB_n$.}\label{fig:Jn}
\end{figure}
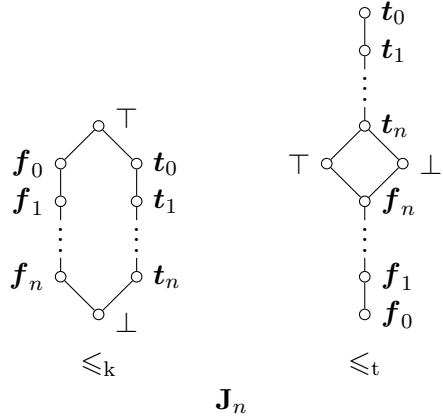

A unary involutive operation $\neg$ that preserves the $\le_\k$-order
and reverses the $\le_\t$-order on $J_n$ is given by:
\[
\neg\top = \top, \quad \neg\bot = \bot, \quad  \neg\mbf_m = \mbt_m \text{ and } \neg\mbt_m = \mbf_m, \text{ for all } m\in [0, n].
\]
We then add every element of $J_n$ as a constant to obtain the prioritised default bilattice
\[
\JB_n = \langle J_n; \otimes, \oplus, \wedge,\vee, \neg, \top,
\mbf_0, \dots, \mbf_n,\mbt_0, \dots, \mbt_n, \bot \rangle,
\]
where $\otimes$ and $\oplus$ are greatest lower bound and least upper bound in the knowledge order $\le_\k$, and $\wedge$ and $\vee$ are greatest lower bound and least upper bound in the truth order $\le_\t$.
\end{df}

Belnap's four-element bilattice, $\FOUR$, is isomorphic to $\JB_0$, and the bilattice $\JB_n$ is obtained from Belnap's bilattice by replacing each of the truth values $\mbf$ and $\mbt$ by chains of truth values of length~$n$.

For all $n\in \omega$, let $\CV_n =\HSP(\JB_n)$ be the variety generated by $\JB_n$. We now quote from~\cite{CDH19} the description of the subdirectly irreducible algebras in~$\CV_n$.

\begin{df}\label{def:M0}
Fix $n\in \omega$. Let $\M_0$ be the algebra in the signature of $\JB_n$ whose reduct is isomorphic to $\JB_0$, as shown in Figure~\ref{fig:M0}, and in which
\begin{itemize}[itemsep=0pt] 

\item the constants $\mbf_0,\mbf_1,\ldots,\mbf_n$ take the value~$\mbf^0$, and

\item the constants
$\mbt_0,\mbt_1, \ldots, \mbt_n$ take the value~$\mbt^0$.
\end{itemize}

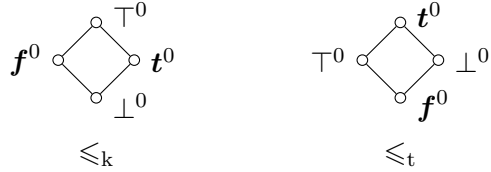
\begin{figure}[hbt]
\centering
\begin{tikzpicture}
\begin{scope}[scale=0.5,xshift=8cm]
  \node at (0,-1.5) {$\le_\t$};
  \node[unshaded] (bot) at (0,0) {};
  \node[unshaded] (f) at (-1,1) {};
  \node[unshaded] (t) at (1,1) {};
  \node[unshaded] (top) at (0,2) {};
  \draw[order] (bot) -- (f) -- (top);
  \draw[order] (bot) -- (t) -- (top);
  \node[label,anchor=west,yshift=-3pt] at (bot) {$\mbf^0$};
  \node[label,anchor=east] at (f) {$\top^0$};
  \node[label,anchor=west] at (t) {$\bot^0$};
  \node[label,anchor=west,yshift=3pt] at (top) {$\mbt^0$};
\end{scope}
\begin{scope}[scale=0.5]
  \node at (0,-1.5) {$\le_\k$};
  \node[unshaded] (bot) at (0,0) {};
  \node[unshaded] (f) at (-1,1) {};
  \node[unshaded] (t) at (1,1) {};
  \node[unshaded] (top) at (0,2) {};
  \draw[order] (bot) -- (f) -- (top);
  \draw[order] (bot) -- (t) -- (top);
  \node[label,anchor=west,yshift=-3pt] at (bot) {$\bot^0$};
  \node[label,anchor=east] at (f) {$\mbf^0$};
  \node[label,anchor=west] at (t) {$\mbt^0$};
  \node[label,anchor=west,yshift=3pt] at (top) {$\top^0$};
\end{scope}
\end{tikzpicture}
\caption{The $\JB_0$-reduct of $\M_0$.} \label{fig:M0}
\end{figure}
\end{df}

\begin{df}\label{def:Mk}
Let $n\in \omega\setminus\{0\}$ and let $k\in [1, n]$. Define $\M_k$ to be the algebra in the signature of $\JB_n$ that has bilattice reduct
isomorphic to the bilattice reduct of~$\JB_1$, as shown in Figure~\ref{fig:J1nm}, and in which
\begin{itemize}[itemsep=0pt] 

\item the constants $\mbf_0,\ldots,\mbf_{k-1}$ take the value $\zero^k$  and the constants $\mbf_k, \ldots, \mbf_n$ take the value~$\mbf^k$, while

\item the constants $\mbt_0,\ldots,\mbt_{k-1}$ take the value~$\one^k$  and the constants $\mbt_k, \ldots, \mbt_n$ take the value~$\mbt^k$.
\end{itemize}
Clearly, $\M_0$ is term equivalent to $\JB_0$, and
$\M_k$ is term equivalent to $\JB_1$, for all $k\in[1, n]$.

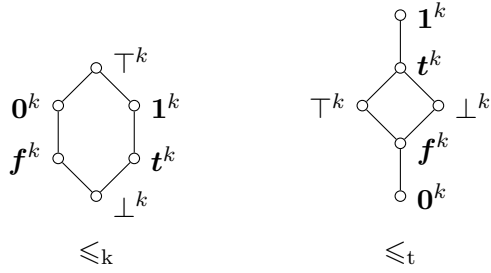
\begin{figure}[ht]
\centering
\begin{tikzpicture}
\begin{scope}[scale=0.5]
  \node at (0,-1.5) {$\le_\k$};
  \node[unshaded] (bot) at (0,0) {};
  \node[unshaded] (f) at (-1,1) {};
  \node[unshaded] (t) at (1,1) {};
  \node[unshaded] (0) at (-1,2.4) {};
  \node[unshaded] (1) at (1,2.4) {};
  \node[unshaded] (top) at (0,3.4) {};
  \draw[order] (bot) -- (f) -- (0) -- (top);
  \draw[order] (bot) -- (t) -- (1) -- (top);
  \node[label,anchor=west,yshift=-3pt] at (bot) {$\bot^k$};
  \node[label,anchor=east] at (f) {$\mbf^k$};
  \node[label,anchor=west] at (t) {$\mbt^k$};
  \node[label,anchor=east] at (0) {$\zero^k$};
  \node[label,anchor=west] at (1) {$\one^k$};
  \node[label,anchor=west,yshift=3pt] at (top) {$\top^k$};
\end{scope}
\begin{scope}[scale=0.5,xshift=8cm]
  \node at (0,-1.5) {$\le_\t$};
  \node[unshaded] (0) at (0,0) {};
  \node[unshaded] (f) at (0,1.4) {};
  \node[unshaded] (top) at (-1,2.4) {};
  \node[unshaded] (bot) at (1,2.4) {};
  \node[unshaded] (t) at (0,3.4) {};
  \node[unshaded] (1) at (0,4.8) {};
  \draw[order] (0) -- (f) -- (top) -- (t) -- (1);
  \draw[order] (f) -- (bot) -- (t);
  \node[label,anchor=west] at (bot) {$\bot^k$};
  \node[label,anchor=west,yshift=-2pt] at (f) {$\mbf^k$};
  \node[label,anchor=west,yshift=2pt] at (t) {$\mbt^k$};
  \node[label,anchor=west] at (0) {$\zero^k$};
  \node[label,anchor=west] at (1) {$\one^k$};
  \node[label,anchor=east] at (top) {$\top^k$};
\end{scope}
\end{tikzpicture}
\caption{The  $\JB_1$-reduct of $\M_k$.}\label{fig:J1nm}
\end{figure}
\end{df}

\begin{prop}[{\cite[Prop.~2.5]{CDH19}}]\label{prop:thevariety}\ 
\begin{enumerate}[label=\textup{(\arabic*)},itemsep=0pt]

\item Up to isomorphism, the only subdirectly irreducible algebra in the variety $\CV_0=\HSP(\JB_0)$ is $\JB_0$ itself.

\item Let $n\in \omega\comp\{0\}$. Up to isomorphism,
the variety $\CV_n=\HSP(\JB_n)$ contains $n+1$ subdirectly irreducible algebras, namely the four-element algebra $\M_0$ and the six-element algebras $\M_k$, for $k\in [1, n]$.

\end{enumerate}

\end{prop}

\section{A natural duality for the variety $\CV_n$}\label{sec:CVduality}

Fix $n\in \omega\comp\{0\}$. We now present the natural duality for the variety $\CV_n$ developed in our companion paper~\cite{CDH19}. We refer to~\cite{CDH19} for a fuller discussion and to Clark and Davey~\cite{CD98} and Davey and Talukder~\cite{DT-Fun} for all of the missing general details.

It follows from Proposition~\ref{prop:thevariety} that
\[
\CV_n = \ISP(\{\M_0, \dots, \M_n\});
\]
so it is natural to use a multi-sorted alter ego based on $\{\M_0,
\dots, \M_n\}$. We use a multi-sorted structure of the following kind:
\[
\MT_n = \langle M_0\du \cdots \du M_n; \mathcal G_{(n)}, \mathcal {S}_{(n)}, \Tp\rangle,
\]
where,
\begin{itemize}[itemsep=0pt] 

\item for each $g\in \mathcal G_{(n)}$, there exist $j, k \in [0, n]$, such that ${g\colon \M_j \to \M_k}$ is a homomorphism,

\item each relation $R\in
\mathcal {S}_{(n)}$ is a compatible relation from $\M_j$~to~$\M_k$, that is, a non-empty subuniverse of $\M_j \times \M_k$, for some
$j, k \in [0, n]$, and

\item $\Tp$ is the disjoint union topology obtained from the discrete topology on the sorts.
\end{itemize}
The precise choices of the sets $\mathcal G_{(n)}$ and $\mathcal {S}_{(n)}$ will be given in Theorem~\ref{cor:bigmultiduality} below.

Objects in the dual category are multi-sorted Boolean
topological structures $\X$ in the signature of~$\MT_n$. Thus,
\[
\X = \langle X_0 \du \cdots \du X_n ; \mathcal G_{(n)}^\X, \mathcal {S}_{(n)}^\X,
\Tp^\X\rangle,
\]
where each $X_j$ carries a Boolean (that is, compact and totally disconnected)
topology and
$\Tp^\X$ is the corresponding disjoint-union topology. If $g\colon \M_j\to \M_k$ is in $\mathcal G_{(n)}$, then the corresponding $g^\X\in
\mathcal G^\X$  is a continuous map $g^\X \colon X_j \to X_k$. 
If $R\in \mathcal S_{(n)}$ is a relation from $\M_j$ to $\M_k$, then the
corresponding $R^\X\in \mathcal S_{(n)}^\X$ is a topologically closed
subset of $X_j\times X_k$. (Typically, we will drop the superscripts from $\mathcal G_{(n)}^\X$, $\mathcal {S}_{(n)}^\X$ and $\Tp^\X$.) Given two such multi-sorted topological
structures $\X$ and $\Y$, a morphism $\varphi\colon \X \to \Y$ is a
continuous map that preserves sorts (so $\varphi(X_j) \subseteq
Y_j$, for all~$j$) and preserves the operations and relations.

For a non-empty set $S$, the power $\MT_n^S$ is defined in the natural sort-wise way; the underlying set of $\MT_n^S$ is $M_0^S \du \dots \du M_n^S$ and the operations and relations between the sorts are defined pointwise. The potential dual category is now defined to be the category
$\CX_n =\IScP(\MT_n)$ whose objects are isomorphic copies of topologically closed substructures of non-zero powers of~$\MT_n$ (where substructure has its natural multi-sorted meaning) and whose morphisms are the continuous structure-preserving maps.

Given an algebra $\A \in \CV_n$, its dual $\mathrm{D}(\A)\in \CX_n$ is defined to be
\[
\mathrm{D}(\A) := \CV_n(\A, \M_0) \du \dots \du \CV_n(\A, \M_n),
\]
as a closed substructure of $\MT_n^A$.
An important, but easily proved, fact is that, for every non-empty set~$S$, we have $\mathrm D((\F_{\CV_n}(S)) \cong \MT_n^S$ (see~\cite[Lemma~2.2.1 and p.~194]{CD98}).
Given $\X\in \CX_n$, its dual $\mathrm{E}(\X) \in \CV_n$ is defined to be
\[
\mathrm{E}(\X) := \CX_n(\X, \MT_n) \text{ as a subalgebra of } \M_0^{X_0} \times \dots \times \M_n^{X_n}.
\]
The fact that the structure on $\MT_n$ is compatible with the set
$\{\M_0,\ldots, \M_n\}$ guarantees that $\mathrm{E}(\X)$ is a
subalgebra of $\M_0^{X_0} \times \dots \times \M_n^{X_n}$ and hence
$\mathrm{E}$ is well defined. The functors ${\mathrm{D}\colon \CV_n \to \CX_n}$ and $\mathrm{E}\colon \CX_n \to
\CV_n$ are defined on homomorphisms in $\CV_n$ and on morphisms in $\CX_n$ in a completely natural way, yielding a dual adjunction $\langle \mathrm{D}, \mathrm{E}, e, \varepsilon\rangle$ between $\CV_n$ and $\CX_n$ with the unit $e_\A
\colon \A \to \mathrm{ED}(\A)$ and counit $\varepsilon_\X \colon \X
\to \mathrm{DE}(\X)$ given by evaluation. 
We say that $\MT_n$
\emph{yields a duality} (between $\CV_n$ and $\CX_n$) if $e_\A$ is
an isomorphism, for all $\A\in \CV_n$, that $\MT_n$ \emph{yields a
full duality} if, in addition, $\varepsilon_\X$ is an isomorphism,
for all $\X\in \CX_n$, and that $\MT_n$ \emph{yields a strong
duality} if $\MT_n$ yields a full duality and is injective
in~$\CX_n$. The duality is called \emph{optimal} if none of the
operations and relations in $\mathcal G_{(n)}\cup\mathcal R_{(n)}$
can be removed without destroying the duality, that is, if an
operation or relation in $\mathcal G_{(n)}\cup\mathcal R_{(n)}$ were
removed, then there would be an algebra $\A\in \CV_n$ such that
$e_\A\colon \A \to \mathrm{DE}(\A)$ is not an isomorphism (or
equivalently, is not surjective).

The duality established in our companion paper~\cite{CDH19}
uses the following multi-sorted operations and relations. For all $k\in [1, n]$, let $g_k \colon \M_j \to \M_0$ be the homomorphism that maps $\mbf^k$ and $\zero^k$ to $\mbf^0$ and maps $\mbt^k$ and $\one^k$ to $\mbt^0$, and for later convenience let $g_0$ be the identity map on~$\M_0$. Define $\lez$, $\lek$ (for $k\in [1, n]$), and $\lejk$ (for $j, k\in [1, n]$ with $j < k$) by
\begin{equation}\tag{$\dagger$}\label{lezetc}
\begin{aligned}
{\lez} &= \big(M_0 \times \{\top^0\}\big) \cup \big(\{\bot^0\} \times M_0\big) \cup \{(\mbf^0, \mbf^0), (\mbt^0, \mbt^0) \}, \\
{\lek} &= \{(\top^k, \top^k), (\bot^k, \bot^k) \} \cup \big(\{(\mbf^k, \mbf^k), (\mbf^k, \zero^k), (\zero^k, \zero^k)\}\big)\\
&\hspace*{5.0cm} \cup
\big(\{ (\mbt^k, \mbt^k), (\mbt^k,\one^k), (\one^k, \one^k)\}\big),\\
{\lejk} &= \{(\top^j, \top^k), (\bot^j, \bot^k)\} \cup \big(\{(\mbf^j, \mbf^k), (\mbf^j, \zero^k), (\zero^j, \zero^k)\}\big)\\
&\hspace*{5.0cm} \cup
\big(\{ (\mbt^j, \mbt^k), (\mbt^j,\one^k), (\one^j, \one^k)\}\big).
\end{aligned}
\end{equation}
Each of these relations is a compatible relation from $\M_i$ to~$\M_\ell$, for some $i\le \ell$ in $[0, n]$.
For all $k\in [0, n]$, the relation $\lek$ is an order; in particular, $\lez$ is the knowledge order on~$\M_0$. The relation $\lejk$ can be thought of as the order relation $\lej$ `stretched' from $M_j$ to $M_k$.  See Figure~\ref{fig:orders}.

\begin{figure}[ht]
\centering
\begin{tikzpicture}[scale=0.6]
\begin{scope}
  \node[anchor=base] at (0,-3.0) {$\lez$};
  \node[unshaded] (bot) at (0,0) {};
  \node[unshaded] (f) at (-1,1) {};
  \node[unshaded] (t) at (1,1) {};
  \node[unshaded] (top) at (0,2) {};
  \draw[order] (bot) -- (f) -- (top);
  \draw[order] (bot) -- (t) -- (top);
  \node[label,anchor=north] at (bot) {$\bot^0$};
  \node[label,anchor=east] at (f) {$\mbf^0$};
  \node[label,anchor=west] at (t) {$\mbt^0$};
  \node[label,anchor=south] at (top) {$\top^0$};
\end{scope}
\begin{scope}[xshift=6cm]
    \node[anchor=base] at (-0.25,-3.0) {$\lej$};
  \node[unshaded] (topj) at (0,3.5) {};
  \node[unshaded] (botj) at (0,-1.5) {};
  \node[unshaded] (fj) at (0,-0.5) {};
  \node[unshaded] (0j) at (0,0.5) {};
  \node[unshaded] (tj) at (0,1.5) {};
  \node[unshaded] (1j) at (0,2.5) {};
  \draw[order] (fj) -- (0j);
  \draw[order] (tj) -- (1j);
  \node[label,anchor=east] at (botj) {$\bot^j$};
  \node[label,anchor=east] at (fj) {$\mbf^j$};
  \node[label,anchor=east] at (tj) {$\mbt^j$};
  \node[label,anchor=east] at (0j) {$\zero^j$};
  \node[label,anchor=east] at (1j) {$\one^j$};
  \node[label,anchor=east] at (topj) {$\top^j$};
\end{scope}
\begin{scope}[xshift=9cm]
    \node[anchor=base] at (0.25,-3.0) {$\lek$};
  \node[unshaded] (topk) at (0,3.5) {};
  \node[unshaded] (botk) at (0,-1.5) {};
  \node[unshaded] (fk) at (0,-0.5) {};
  \node[unshaded] (0k) at (0,0.5) {};
  \node[unshaded] (tk) at (0,1.5) {};
  \node[unshaded] (1k) at (0,2.5) {};
  \draw[order] (fk) -- (0k);
  \draw[order] (tk) -- (1k);
  \node[label,anchor=west] at (botk) {$\bot^k$};
  \node[label,anchor=west] at (fk) {$\mbf^k$};
  \node[label,anchor=west] at (tk) {$\mbt^k$};
  \node[label,anchor=west] at (0k) {$\zero^k$};
  \node[label,anchor=west] at (1k) {$\one^k$};
  \node[label,anchor=west] at (topk) {$\top^k$};
\end{scope}
\begin{scope}[xshift=3cm]
    \node[anchor=base] at (0,-3.0) {$g_j$};
  \path (topj) edge [operationgj, bend right=15] node {} (top);
  \path (1j) edge [operationgj, bend right=15] node {} ($(t)+(3pt,10pt)$);
  \path (tj) edge [operationgj, bend right=15] node {} ($(t)+(8pt,10pt)$);
  \path (0j) edge [operationgj, bend left=15] node {} ($(f)+(3pt,3pt)$);
  \path (fj) edge [operationgj, bend left=15] node {} (f);
  \path (botj) edge [operationgj, bend left=15] node {} (bot);
\end{scope}
\begin{scope}[xshift=7.5cm]
    \node[anchor=base] at (0,-3.0) {$\lejk$};
  \draw[relationlejk] (topj) -- (topk);
  \draw[relationlejk] (botj) -- (botk);
  \draw[relationlejk] (fj) -- (fk);
  \draw[relationlejk] (fj) -- (0k);
  \draw[relationlejk] (0j) -- (0k);
  \draw[relationlejk] (tj) -- (tk);
  \draw[relationlejk] (tj) -- (1k);
  \draw[relationlejk] (1j) -- (1k);
\end{scope}
\end{tikzpicture}
\caption{The map $g_j$ and the relations $\lez$, $\lek$ and $\lejk$.}\label{fig:orders}
\end{figure}
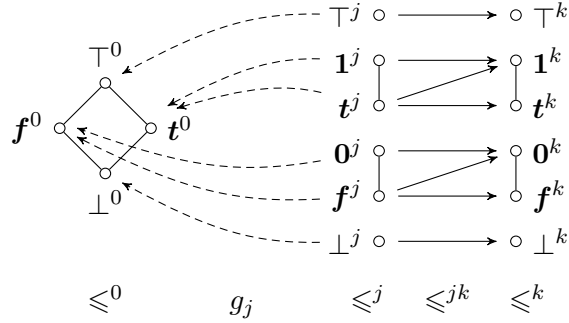

We can now state the main result of~\cite{CDH19}.

\begin{thm}[{\cite[Theorem 4.1]{CDH19}}]\label{cor:bigmultiduality}
Let $n\in \omega\comp\{0\}$. Define the multi-sorted alter ego
\[
\MT_n = \langle M_0\du \cdots \du M_n;
\mathcal G_{(n)}, \mathcal{S}_{(n)}, \Tp\rangle,
\]
where
\begin{align*}
\mathcal{G}_{(n)} &= \big\{\,g_k \mid k\in [1, n]\,\big\}, \text{ and}\\
\mathcal S_{(n)} &= \{\,{\lek} \mid k\in [0, n]\,\} \cup \big\{\,{\lejk} \mid j, k\in [1, n] \text{ with } j < k\,\big\}.
\end{align*}
The alter ego $\MT_n$ yields a strong, and therefore full,
optimal duality between $\CV_n = \HSP(\JB_n)$ $= \ISP(\{\M_0,
\dots, \M_n\})$ and $\CX_n = \IScP(\MT_n)$. 
\end{thm}

\section{A description of the objects in $\CX_n$}\label{sec:dualcategory}

Our aim in this section is to give an intrinsic description of the objects in~$\CX_n$.
We require the  multi-sorted version of the Separation Lemma~\cite[Theorem 1.4.4]{CD98}.  We shall state it only in the case of topological structures in the signature of~$\MT_n$.

\begin{lem}\label{lem:sep}
Let $n\in \omega\comp\{0\}$. For a multi-sorted compact topological structure
\[
\X = \langle X  ; \mathcal G_{(n)}, \mathcal {S}_{(n)}, \Tp\rangle, \text{ where } X = X_0 \du \cdots \du X_n,
\]
in the signature of $\MT_n$, we have $\X\in \IScP(\MT_n)$ if and only if
\begin{enumerate}[label=\textup{(\arabic*)},leftmargin=1.25\parindent,itemsep=0pt] 

\item there is a morphism $\varphi\colon \X \to \MT_n$,

\item the morphisms from $\X$ to $\MT_n$ separate the elements of each sort, that is, for all $k\in [0,n]$ and all $x, y \in X_k$ with $x\ne y$, there exists a morphism $\varphi_e \colon \X \to \MT_n$ satisfying $\varphi_e(x) \ne \varphi_e(y)$,

\item the morphisms from $\X$ to $\MT_n$ separate each relation in $\mathcal {S}_{(n)}$, that is,

\begin{enumerate}[label=\textup{(\alph*)},leftmargin=\parindent,itemsep=0pt] 

\item for all $k\in [0,n]$ and all $x\nleqslant^k y$ in $X_k$,
there exists a morphism $\varphi_k \colon \X \to \MT_n$ satisfying $\varphi_k(x) \nlek  \varphi_k(y)$,

\item for all $j, k\in [1, n]$ with $j < k$,
and all $x\in X_j$ and $y\in X_k$ with $x \nlejk y$, there exists a morphism $\varphi_{jk} \colon \X \to \MT_n$ satisfying $\varphi_{jk}(x) \nlejk  \varphi_{jk}(y)$.

\end{enumerate}

\end{enumerate}
\end{lem}

\goodbreak

We now give our axiomatisation of the category~$\CX_n$.
Axiom~\ref{A1} requires the operations to be continuous; Axioms~\ref{A2}--\ref{A5} are first order, indeed, they are quasi-equational; Axioms~\ref{A6} and~\ref{A7} are topological separation axioms. We do not require the relations to be topologically closed in the appropriate product spaces as this follows from the axioms; indeed, if $\X$ satisfies the axioms, then the theorem implies that $\X$ is isomorphic both structurally and topologically to a closed substructure of a power of~$\MT_n$, whence the relations on $\X$ must be topologically closed.

Let $n \in \omega\comp\{0\}$ and let $\X = \langle X  ; \mathcal G_{(n)}, \mathcal {S}_{(n)}, \Tp\rangle$, where $X = X_0 \du \cdots \du X_n$,
be a multi-sorted topological structure in the signature of $\MT_n$. As it will allow some of our arguments to be more compact, we define ${\lekk} := {\lek}$, for each $k \in [1, n]$.

Let $j, k\in [1, n]$ with $j \le k$ and for each $\ell \in [j, k]$ let $U_\ell$ be a subset of $X_\ell$. We say that $U_j, \dots, U_k$ are \emph{mutually increasing up-sets} if, for all $i, \ell\in [j, k]$ with $i \le \ell$, whenever $x\in U_i$ and $y\in X_\ell$ with $x \leil y$, we have $y\in U_\ell$.
If $\lel$ is an order, then by taking $i = \ell$, this definition guarantees that $U_\ell$ is indeed an up-set in $\langle X_\ell; \lel\rangle$.

Recall that an ordered topological space $\langle X; \le,
\Tp\rangle$ is a \emph{Priestley space} if it is compact and
totally order-disconnected,  that is, for all $x,y \in X$
with $x\nle y$, there exists a clopen up-set $U$ with $x\in U$ and
$y\notin U$.
Note that the underlying topological space of a
Priestley space is Boolean.

\begin{thm}\label{thm:dualcategory}
Let $n\in \omega\comp\{0\}$. For a multi-sorted topological structure
\[
\X = \langle X  ; \mathcal G_{(n)}, \mathcal {S}_{(n)}, \Tp\rangle, \text{ where } X = X_0 \du \cdots \du X_n,
\]
in the signature of $\MT_n$, we have $\X \in \IScP(\MT_n)$ if and only if

\begin{enumerate}[label=\textup{(A\arabic*)},leftmargin=1.7\parindent,itemsep=0pt] 

\item \label{A1} $g_k\colon X_k \to X_0$ is continuous, for all $k\in [1,n]$;

\item \label{A2} for all $k\in [1,n]$ and all $x,y\in X_k$, if $x \lek y$, then $g_k(x)= g_k(y)$;

\item \label{A3} for all $j, k\in [1,n]$ with $j < k$ and all $x\in X_j$ and $y\in X_k$,
if $x \lejk y$, then $g_j(x)= g_k(y)$;

\item \label{A4} for all $j, k\in [1,n]$ with $j < k$ and  all $x,y\in X_j$ and $u,v\in X_k$,
if $x \lej y$ and $y \lejk u$ and $u \lek v$, then $x \lejk v$;

\item \label{A5} for all $j, k, \ell\in [1, n]$ with $j < k < \ell$ and  all $x\in X_j$, $y\in X_k$ and $z\in X_\ell$,
if $x \lejk y$ and $y \lekl z$, then $x \lejl z$;

\item \label{A6} $\langle X_k; \lek,{\Tp}_k \rangle$ is a Priestley space for all $k\in [0, n]$, where ${\Tp}_k$ is the topology induced by $\Tp$ \textup(in particular, $\langle X_k; \lek\rangle$ is an ordered set\textup);

\item \label{A7} for all $j, k\in [1,n]$ with $j < k$ and  all $x\in X_j$ and $y\in X_k$ with ${x \nlejk y}$, there exist $U_j, U_{j+1}, \dots, U_k$, with $U_\ell$ a clopen up-set in $\langle X_\ell; \lel,{\Tp}_\ell \rangle$, for all $\ell\in[j, k]$, such that $U_j, \dots, U_k$ are mutually increasing with $x\in U_j$ and $y\in X_k\comp U_k$.

\end{enumerate}
Note that when $n = 1$, only conditions \ref{A1}, \ref{A2} and \textup{\ref{A6}} apply.
\end{thm}

Before giving the proof of Theorem~\ref{thm:dualcategory}, we will refine conditions~\ref{A6} and~\ref{A7} into four conditions. This will enable us to remove some repetition from the proof.

\begin{lem}
Given that~\ref{A4} and \ref{A5} hold, \ref{A6} and \ref{A7} are together equivalent to the following four conditions
\begin{enumerate}[label=\textup{(A\arabic*)},leftmargin=2.0\parindent,itemsep=0pt] 

\item[\upshape(A6)$^0$] $\langle X_0; \lez, {\Tp}_0 \rangle$ is a Priestley space, where ${\Tp}_0$ is the topology induced by~$\Tp$,

\item[\upshape(A6)$'$] $\lek$ is an order on $X_k$, for all $k\in [1, n]$,

\item[\upshape(A6)$''$] $\langle X_k; {\Tp}_k \rangle$ is a compact space for all $k\in [1, n]$, where ${\Tp}_k$ is the topology induced by~$\Tp$,

\item[\upshape(A7)$'$] for all $j, k\in [1,n]$ with $j \le k$ and  all $x\in X_j$ and $y\in X_k$ with ${x \nlejk y}$, there exist $U_j, U_{j+1}, \dots, U_k$, with $U_\ell$ a clopen up-set in $\langle X_\ell; \lel,{\Tp}_\ell \rangle$, for all $\ell\in[j, k]$, such that $U_j, \dots, U_k$ are mutually increasing with $x\in U_j$ and $y\in X_k\comp U_k$.

\end{enumerate}
\end{lem}

\begin{proof}
This is straightforward. The main thing to note is that, in the presence of \ref{A6}$'$ and \ref{A6}$''$, when $j = k$ condition~\ref{A7}$'$ says precisely that $\langle X_k; \lek, {\Tp}_k \rangle$ is totally order-disconnected, and therefore says that $\langle X_k; \lek, {\Tp}_k \rangle$ is a Priestley space.
\end{proof}

We now turn to the proof of the theorem.

\begin{proof}[Proof of Theorem~\ref{thm:dualcategory}]
Let $\X \in \IScP(\MT_n)$. It is easy to see that $\MT_n$ satisfies \ref{A1}--\ref{A6}. Since these conditions
are preserved under multi-sorted products and closed substructures it follows at once that $\X$ satisfies them.

We now show that $\X$ satisfies the topological condition \ref{A7}. Since \ref{A7} is inherited by closed substructures, it suffices to assume that $\X$ is a multi-sorted power~$\MT_n^S$, for some set~$S$. Assume that $j < k$ in $[1, n]$ and  let $x\in M_j^S$ and $y\in M_k^S$ with ${x \nlejk y}$. Then there exists $s\in S$ such that $x(s) \nlejk y(s)$. Define
\[
U_j := \{\, z \in M_j^S \mid x(s) \lej z(s)\,\},
\]
and for $\ell\in [j+1, k]$ define
\[
U_\ell := \{\, z \in M_i^S \mid x(s) \lejl z(s)\,\}.
\]
Then $U_\ell$ is a clopen up-set in $\langle X_\ell; \lel, \Tp_\ell\rangle$, for all $\ell\in [j, k]$. Since $x\in U_j$ and $y\in X_k\comp U_k$, it remains to see that $U_j, U_{j+1}, \dots, U_k$ are mutually increasing. Let $i < \ell$ in $[j, k]$ and assume that $w\in U_i$ and $z\in M_\ell^S$ with $w \leil z$. Hence $x(s) \leji w(s)$, as $w\in U_i$, and $w(s) \leil z(s)$, as $w \leil z$. Then \ref{A5} yields $x(s) \lejl z(s)$, whence $z\in U_\ell$. Thus $U_j, U_{j+1}, \dots, U_k$ are mutually increasing.

For the converse, assume that $\X$ satisfies \ref{A1}--\ref{A7}. By Lemma~\ref{lem:sep}, we must show that there is a morphism $\varphi\colon \X \to \MT_n$, and that the morphisms from $\X$ to $\MT_n$ separate the elements of each sort and separate the relations in $\mathcal S_{(n)}$.

Since the set $T := \{\mbt^0, \mbt^1, \dots, \mbt^n\}$ forms a substructure of $\MT_n$ on which every relation is total, we may define a morphism $\varphi\colon \X \to \MT_n$ by mapping $X_k$ constantly to $\mbt^k$, for all $k\in [0, n]$. 
For all $k\in [0, n]$,  if $x\ne y$ in~$X_k$, then since $\lek$ is an order, either $x \nlek y$ or $y \nlek x$; so being able to separate the relation $\lek$ implies that we can separate the points of~$X_k$.

We next separate the order on $X_0$, and then 
the relation $\lejk$ (from $X_j$ to $X_k$)
for $j, k\in [1, n]$ with $j \le k$.
Let $x,y\in X_0$ with $x\nlez y$. We must define a morphism $\varphi_0\colon \X \to \MT_n$ satisfying $\varphi_0 (x) \nlez  \varphi_0 (y)$.
By \ref{A6}$^0$ there is a clopen up-set $U \subseteq X_0$ containing $x$ but not~$y$. We shall see that we obtain the required  morphism by defining
\[
\varphi_0 (z) := \begin{cases} \top^0 & \quad\text{if $z\in U$},\\
\bot^0 & \quad\text{if $z\in X_0\comp U$},\\
\top^k & \quad\text{if $z\in g_k^{-1}(U)$, for some $k\in [1, n]$},\\
\bot^k & \quad\text{if $z\in g_k^{-1}(X_0\comp U)$, for some $k\in [1, n]$}.
\end{cases}
\]
We first prove that $\varphi_0$ preserves $g_k$, for all $k\in [1, n]$. Let $k\in [1, n]$ and let $z\in X_k$.  By the definitions of $g_k$ and $\varphi_0$, if $g_k(z) \in U$
then $\varphi_0(z) = \top^k$ and so $g_k(\varphi_0(z)) = \top^0 = \varphi_0(g_k(z))$ and if $g_k(z) \notin U$
then $\varphi_0(z) = \bot^k$ and so $g_k(\varphi_0(z)) = \bot^0 = \varphi_0(g_k(z))$. Thus $\varphi_0$ preserves~$g_k$. 
By \ref{A2} and \ref{A3} respectively, $\varphi_0$ is $\lek$-preserving and $\lejk$-preserving.
Finally, $\varphi_0$ is continuous since $U$ is clopen and, by \ref{A1}, the operation  $g_k\colon X_k \to X_0$ is continuous, for all $k\in [1, n]$.
Hence $\varphi_0$ is a morphism from $\X$ to $\MT_n$ such that $\varphi_0 (x) = \top^0 \nlez  \bot^0 = \varphi_0 (y)$.

We now consider the relation $\lejk$ from $X_j$ to $X_k$
for $j, k\in [1, n]$ with $j \le k$. Let $x\in X_j$, $y\in X_k$ with $x \nlejk y$. We must define a morphism $\varphi_{jk}\colon \X \to \MT_n$ satisfying $\varphi_{jk} (x) \nlejk  \varphi_{jk} (y)$.
By~\ref{A7}$'$ there exist mutually increasing clopen up-sets $U_j, \dots, U_k$  such that $x\in U_j$ and $y\in X_k\comp U_k$. We define the mapping $\varphi_{jk}$ by
\[
\varphi_{jk} (z) := \begin{cases} \mbt^\ell & \quad\text{if $z\in X_\ell$ for some $\ell \in [0, j-1]$},\\
\one^{\ell} & \quad\text{if $z\in U_\ell$ for some $\ell\in[j, k]$},\\
\mbt^{\ell} & \quad\text{if $z\in X_{\ell}\comp U_\ell$ for some $\ell\in[j, k]$},\\
\one^\ell & \quad\text{if $z\in X_\ell$  for some $\ell \in [k+1, n]$}.\\
\end{cases}
\]
We first check that $\varphi_{jk}$ preserves $g_\ell$ for all $\ell\in [1, n]$. Let $\ell\in [1, n]$ and let $z\in X_\ell$.  Then $\varphi_{jk} (z) \in \{\one^\ell, \mbt^\ell\}$ and so $g_\ell(\varphi_{jk}(z)) = \mbt^0$. As $g_\ell(z)\in X_0$ we also have $\varphi_{jk}(g_\ell(z)) = \mbt^0$. Hence $\varphi_{jk}$ preserves~$g_\ell$.

Since $\varphi_{jk}$ is constant on $X_0$, it trivially preserves~$\lez$, and it remains to prove that $\varphi_{jk}$ preserves $\leil$ for all $i\le \ell$ in $[1, n]$. Let $i\le \ell$ in $[1, n]$ and let $w\in X_i$ and $z\in X_\ell$ with $w \leil z$. Since $\{(\mbt^i,\mbt^\ell),(\mbt^i,\one^\ell),(\one^i,\one^\ell)\} \subseteq {\leil}$, to prove that $\varphi_{jk}(w) \leil\varphi_{jk}(z)$ it suffices to show that $(\varphi_{jk}(w),\varphi_{jk}(z)) \ne (\one^i, \mbt^\ell)$. Suppose, by way of contradiction, that $(\varphi_{jk}(w),\varphi_{jk}(z)) = (\one^i, \mbt^\ell)$. Then $\varphi_{jk}(w) = \one^i$  so $j \le i$, and
$\varphi_{jk}(z) = \mbt^\ell$  so $\ell \le k$. Hence $i\le \ell$ in $[j, k]$.
Thus $\varphi_{jk}(w) = \one^i$ implies that $w\in U_i$. As $w \leil z$ and $U_j, \dots, U_k$ are mutually increasing, we conclude that $z\in U_\ell$, whence $\varphi_{jk}(z) = \one^\ell$, contradicting our assumption that $\varphi_{jk}(z) = \mbt^\ell$. Hence $\varphi_{jk}$ preserves~$\leil$.

Since $U_\ell$ is clopen in $X_\ell$, for all $\ell \in [j, k]$, and since $\varphi_{jk}$ is constant on all $X_\ell$ with $\ell\notin [j, k]$, it is trivial that $\varphi_{jk}$ is continuous. Therefore $\varphi_{jk}$ is a morphism from $\X$ to $\MT_n$ with $\varphi_{jk} (x) = \one^{j} \nlejk  \mbt^{k} = \varphi_{jk} (y)$, as required.
\end{proof}

\section{An isomorphic category of single-sorted structures}\label{sec:isocat}
Throughout this section, we fix $n \in \omega\comp\{0\}$. We will define a category $\CY_n$ that is isomorphic to $\CX_n$ but consists of topological structures that are single sorted. The objects of $\CY_n$ will be Priestley spaces with a continuous retraction in which the order has a natural ranking.

We first give a lemma that will do much of the work for us.

\begin{lem}\label{lem:lePYonX}
Let $\X = \langle X  ; \mathcal G_{(n)}, \mathcal {S}_{(n)}, \Tp\rangle$, where $X = X_0 \du \cdots \du X_n$, be a multi-sorted topological structure in the signature of $\MT_n$. Define a binary relation $\lePY$ on $X$ by
\[
x \lePY y \iff x\in X_j, y\in X_k \text{ and } x \lejk y, \text{ for some $j, k \in [1, n]$ with } j \le k.
\]

\begin{enumerate}[label=\textup{(\alph*)},itemsep=0pt] 

\item $\X$ satisfies  \ref{A4}, \ref{A5} and \ref{A6}$'$ if and only if $\lePY$ is an order relation on~$X$.

\item Assume that $\X$ satisfies \ref{A4}, \ref{A5} and \ref{A6}$'$. If $U$ is a clopen up-set in ${\langle X; \lePY, \Tp\rangle}$, then $U\cap X_0,
\dots, U\cap X_n$ are mutually increasing clopen up-sets. Conversely, let
$j < k$ in $[0, n]$,
let $U_i\subseteq X_i$, for all $i\in [j, k]$, and assume that $U_j, U_{j+1}, \dots, U_k$ are mutually increasing clopen up-sets. Then there exists a clopen up-set $U$ in ${\langle X; \lePY, \Tp\rangle}$ with $U_i = U \cap X_i$, for all $i\in [j, k]$.

\item $\X$ satisfies Axioms \ref{A4}--\ref{A7} if and only if
$
\langle X; \lePY, \Tp\rangle
$
is a Priestley space.

\end{enumerate}

\end{lem}

\begin{proof}
(a) Assume that $\X$ satisfies \textup{\ref{A4}, \ref{A5}} and \textup{\ref{A6}$'$}. Since $\lek$ is reflexive by \ref{A6}$'$, it follows immediately that $\lePY$ is reflexive. Assume that $x\lePY y$ and $y \lePY x$, for some $x, y\in X$,
say $x\in X_j$ and $y\in X_k$. It follows from the definition of $\lePY$ on $X$ that $j\le k$ and $k\le j$, whence
$x, y\in X_k$, for some $k\in [0, n]$. Hence $x \lek y$ and $y \lek x$, and as $\lek$ is antisymmetric, we have $x = y$.
Hence, $\lePY$ is antisymmetric. Now assume that $x, y, z\in X$ with $x\lePY y$ and $y\lePY z$. Then either
\begin{enumerate}[label=\textup{(\roman*)}, itemsep=0pt] 

\item $x$, $y$ and $z$ belong to the same sort, in which case \ref{A6}$'$ guarantees that $x \lePY z$, or

\item there exist $j < k$ in $[1, n]$ with either $x, y\in X_j$ and $z\in X_k$ or $x\in X_j$ and $y, z\in X_k$, in which case \ref{A4} guarantees that $x \lePY z$, or

\item
there exist $j < k < \ell$ in $[1, n]$ with $x\in X_j$, $y\in X_k$ and $z\in X_\ell$, in which case \ref{A5} guarantees that $x \lePY z$.

\end{enumerate}
Hence $\lePY$ is transitive.

The converse implication is almost immediate as \ref{A4} and \ref{A5} follow from the transitivity of $\lePY$.

(b) If $U$ is a clopen up-set in $\langle X; \lePY, \Tp\rangle$, then $U\cap X_i$ is a clopen in $X_i$ and the definition of $\lePY$ guarantees that $U\cap X_0,
\dots, U\cap X_n$ are mutually increasing up-sets. Now let $j < k$ in $[0, n]$, let $U_i\subseteq X_i$, for all $i\in [j, k]$, and assume that $U_j, U_{j+1}, \dots, U_k$ are mutually increasing clopen up-sets. Then
\[
U:= U_j \cup U_{j+1} \cup \dots\cup U_k\cup X_{k+1}\cup \dots \cup X_n
\]
is a clopen up-set in $\langle X; \lePY, \Tp\rangle$ satisfying $U_i = U \cap X_i$, for all $i\in [j, k]$, as required.

(c) Assume that $\X$ satisfies \ref{A4}--\ref{A7}. By (a) it remains to prove that $\langle X; \lePY, \Tp\rangle$ is totally order-disconnected. Let $x, y \in X$ with $x\nlePY y$. Let $x\in X_j$ and $y\in X_k$. If $j > k$, then $U := X_j \cup X_{j+1} \cup \dots \cup X_n$ is a clopen up-set in $\langle X; \lePY, \Tp\rangle$ containing $x$ but not~$y$. If $j = k$, then, since $\langle X_j; \lej, \Tp_j\rangle$ is a Priestley space, by \ref{A6}, there is a clopen up-set $V$ in $X_j$ containing $x$ but not~$y$. Then $U := V \cup X_{j+1} \cup \dots \cup X_n$ is a clopen up-set in $\langle X; \lePY, \Tp\rangle$ containing $x$ but not~$y$. If $j < k$, then $x\nlePY y$ implies that $x\nlejk y$. Hence by \ref{A7} and~(b), there exists a clopen up-set $U$ of $\langle X; \lePY, \Tp\rangle$ containing $x$ but not~$y$.
Thus $\langle X; \lePY, \Tp\rangle$ is totally order-disconnected.

Conversely, assume that $\langle X; \lePY, \Tp\rangle$ is a Priestley space. By (a), it remains to show that \ref{A6} and \ref{A7} hold. Let $k\in [0, n]$. Since $X_k$ is a closed subset of $\langle X; \lePY, \Tp\rangle$ and since the order and topology on $\langle X_k; \lek,{\Tp}_k \rangle$ are the restrictions of those on $\langle X; \lePY, \Tp\rangle$, it follows at once that $\langle X_k; \lek,{\Tp}_k \rangle$ is a Priestley space. Thus \ref{A6} holds.

Now let $j < k$ in $[1, n]$ and let $x\in X_j$ and $y\in X_k$ with $x \nlejk y$. Hence $x\nlePY y$ and, since $\langle X; \lePY, \Tp\rangle$ is a Priestley space, there exists a clopen up-set $U$ in $\langle X; \lePY, \Tp\rangle$ with $x\in U$ and $y\notin U$. Define $U_i := U\cap X_i$, for all $i\in [j,k]$. By (b), the sets $U_j, U_{j+1}, \dots, U_k$ are mutually increasing clopen up-sets such that $x\in U_j$ and $y\in X_k\comp U_k$, whence \ref{A7} holds.
\end{proof}

\begin{df}
Let  $\X = \langle X_0 \du \cdots \du X_n  ; \mathcal G_{(n)}, \mathcal {S}_{(n)}, \Tp\rangle$ be an object in $\CX_n$. For convenience, we shall define $g_0$ to be the identity function on~$X_0$.
We define a new structure
\[
\mathrm{F}(\X) := \langle X; \lePY, g, \rank, \Tp\rangle
\]
as follows:
\begin{itemize}[itemsep=0pt] 

\item $X := X_0 \du \cdots \du X_n$,

\item $\lePY$ is the binary relation on $X$ defined in Lemma~\ref{lem:lePYonX},

\item   $g\colon X \to X$ is defined by $g(x) := g_k(x)$, for all $x\in X_k$ and all $k\in [0, n]$,

\item $\rank \colon X \to [0,n]$ is defined by $\rank(x) := k$, for all $x\in X_k$ and for all $k\in [0, n]$,

\end{itemize}
\end{df}

For example, $\mathrm{F}(\MT_n) = \langle M_0 \du \dots\du M_n; \lePY, g, \rank, \Tp\rangle$ is illustrated in Figure~\ref{fig:FofMn}.

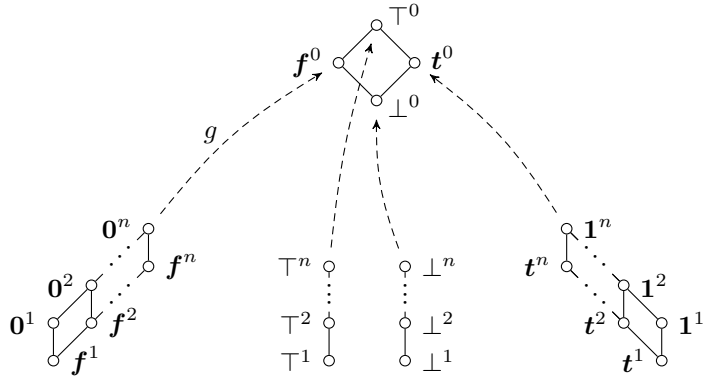
\begin{figure}[ht]
\centering
\small
\begin{tikzpicture}

\begin{scope}[scale=0.5, xshift=4cm]
  \node[unshaded] (botb) at (0,1.4) {};
  \node[unshaded] (fb) at (-1,2.4) {};
  \node[unshaded] (tb) at (1,2.4) {};
  \node[unshaded] (topb) at (0,3.4) {};
  \draw[order] (botb) -- (fb) -- (topb);
  \draw[order] (botb) -- (tb) -- (topb);

  \node[label,anchor=west,yshift=2pt] at (tb) {$\mbt^{0}$};
  \node[label,anchor=west,yshift=-2pt,xshift=-2pt] at (botb) {$\bot^{0}$};
  \node[label,anchor=west,yshift=4pt,xshift=-2pt] at (topb) {$\top^{0}$};
  \node[label,anchor=east] at (fb) {$\mbf^{0}$};
\end{scope}

\begin{scope}[scale=0.5, xshift=2.75cm, yshift=-4cm]

  \node[unshaded] (topg1) at (0,1){};
  \node[unshaded] (topg2) at (0,-1.5){};
  \node[unshaded] (topg3) at (0,-0.5){};

  \path (topg1)  edge [map, bend left=5] node {} (topb);

  \node at ($0.35*(topg1)+0.65*(topg3)$) {$\cdot$};
  \node at ($0.5*(topg1)+0.5*(topg3)$) {$\cdot$};
  \node at ($0.65*(topg1)+0.35*(topg3)$) {$\cdot$};
  \draw (topg1) -- ($0.8*(topg1)+0.2*(topg3)$);
  \draw (topg3) -- ($0.2*(topg1)+0.8*(topg3)$);
  \draw[order] (topg2) -- (topg3);

  \node[label,anchor=east] at (topg1) {$\top^{n}$};
  \node[label,anchor=east] at (topg2) {$\top^1$};
  \node[label,anchor=east] at (topg3) {$\top^2$};
\end{scope}

\begin{scope}[scale=0.5, xshift=4.75cm, yshift=-4cm]

  \node[unshaded] (botg1) at (0,1){};
  \node[unshaded] (botg2) at (0,-1.5){};
  \node[unshaded] (botg3) at (0,-0.5){};

  \path (botg1)  edge [map, bend left=10] node {} (botb);
  \node at ($0.35*(botg1)+0.65*(botg3)$) {$\cdot$};
  \node at ($0.5*(botg1)+0.5*(botg3)$) {$\cdot$};
  \node at ($0.65*(botg1)+0.35*(botg3)$) {$\cdot$};

  \draw (botg1) -- ($0.8*(botg1)+0.2*(botg3)$);
  \draw (botg3) -- ($0.2*(botg1)+0.8*(botg3)$);
  \draw[order] (botg2) -- (botg3);

  \node[label,anchor=west] at (botg1) {$\bot^{n}$};
  \node[label,anchor=west] at (botg2) {$\bot^1$};
  \node[label,anchor=west] at (botg3) {$\bot^2$};
\end{scope}

\begin{scope}[scale=0.5, xshift=-2cm, yshift=-3cm]

  \node[unshaded] (fg1) at (0,0) {};
  \node[unshaded] (0g1) at (0,1){};
  \node[unshaded] (0gz) at (-2.5,-1.5) {};
  \node[unshaded] (fgz) at (-2.5,-2.5){};
  \node[unshaded] (0gzz) at (-1.5,-0.5) {};
  \node[unshaded] (fgzz) at (-1.5,-1.5){};

  \draw[order] (fg1) -- (0g1);
  \draw[order] (fgz) -- (0gz);
  \draw[order] (fgz) -- (fgzz);
  \draw[order] (0gz) -- (0gzz);
  \draw[order] (fgzz) -- (0gzz);
  \node at ($0.35*(fg1)+0.65*(fgzz)$) {$\cdot$};
  \node at ($0.5*(fg1)+0.5*(fgzz)$) {$\cdot$};
  \node at ($0.65*(fg1)+0.35*(fgzz)$) {$\cdot$};
  \node at ($0.35*(0g1)+0.65*(0gzz)$) {$\cdot$};
  \node at ($0.5*(0g1)+0.5*(0gzz)$) {$\cdot$};
  \node at ($0.65*(0g1)+0.35*(0gzz)$) {$\cdot$};
  \draw (fg1) -- ($0.8*(fg1)+0.2*(fgzz)$);
  \draw (0g1) -- ($0.8*(0g1)+0.2*(0gzz)$);
  \draw (fgzz) -- ($0.2*(fg1)+0.8*(fgzz)$);
  \draw (0gzz) -- ($0.2*(0g1)+0.8*(0gzz)$);

  \path (0g1)  edge [map, bend left=10] node {} (fb);
  \node[label,anchor=east] at (2.25,3.5) {$g$};

  \node[label,anchor=east] at (0g1) {$\zero^{n}$};
  \node[label,anchor=west] at (fg1) {$\mbf^{n}$};
  \node[label,anchor=east] at (0gz) {$\zero^1$};
  \node[label,anchor=west] at (fgz) {$\mbf^1$};
  \node[label,anchor=east] at (0gzz) {$\zero^2$};
  \node[label,anchor=west] at (fgzz) {$\mbf^2$};

  \end{scope}

\begin{scope}[scale=0.5, xshift=9cm, yshift=-3cm]

  \node[unshaded] (tg1) at (0,0) {};
  \node[unshaded] (1g1) at (0,1){};
  \node[unshaded] (1gz) at (2.5,-1.5) {};
  \node[unshaded] (tgz) at (2.5,-2.5){};
  \node[unshaded] (1gzz) at (1.5,-0.5) {};
  \node[unshaded] (tgzz) at (1.5,-1.5){};

  \draw[order] (tg1) -- (1g1);
  \draw[order] (1gz) -- (1gzz);
  \draw[order] (tgz) -- (tgzz);
  \draw[order] (tgz) -- (1gz);
 \draw[order] (tgzz) -- (1gzz);
  \node at ($0.35*(tg1)+0.65*(tgzz)$) {$\cdot$};
  \node at ($0.5*(tg1)+0.5*(tgzz)$) {$\cdot$};
  \node at ($0.65*(tg1)+0.35*(tgzz)$) {$\cdot$};
  \node at ($0.35*(1g1)+0.65*(1gzz)$) {$\cdot$};
  \node at ($0.5*(1g1)+0.5*(1gzz)$) {$\cdot$};
  \node at ($0.65*(1g1)+0.35*(1gzz)$) {$\cdot$};
  \draw (tg1) -- ($0.8*(tg1)+0.2*(tgzz)$);
  \draw (1g1) -- ($0.8*(1g1)+0.2*(1gzz)$);
  \draw (tgzz) -- ($0.2*(tg1)+0.8*(tgzz)$);
  \draw (1gzz) -- ($0.2*(1g1)+0.8*(1gzz)$);
  \draw (tg1) -- ($0.8*(tg1)+0.2*(tgzz)$);
  \draw (1g1) -- ($0.8*(1g1)+0.2*(1gzz)$);
  \draw (tgzz) -- ($0.2*(tg1)+0.8*(tgzz)$);
  \draw (1gzz) -- ($0.2*(1g1)+0.8*(1gzz)$);

  \path (1g1)  edge [map, bend right=10] node {} (tb);

  \node[label,anchor=west] at (1g1) {$\one^{n}$};
  \node[label,anchor=east] at (tg1) {$\mbt^{n}$};
  \node[label,anchor=west] at (1gz) {$\one^1$};
  \node[label,anchor=east] at (tgz) {$\mbt^1$};
  \node[label,anchor=west] at (1gzz) {$\one^2$};
  \node[label,anchor=east] at (tgzz) {$\mbt^2$};

  \end{scope}

\end{tikzpicture}
\caption{The structure $\mathrm{F}(\MT_n)$.}\label{fig:FofMn}
\end{figure}

The next lemma gives properties of  $\mathrm{F}(\X)$ that will be used to define the category~$\CY_n$. Given $n\in \omega\comp\{0\}$, we define an \emph{$n$-ranking} of a Priestley space
$\langle X; \lePY, \Tp\rangle$ to be a continuous order-preserving map, $\rank$, from $\langle X; \lePY, \Tp\rangle$ to the finite Priestley space $\langle[0, n]; \le, \Tp_d\rangle$, where $\le$ is the usual order inherited from~$\mathbb Z$ and $\Tp_d$ is the discrete topology.

\begin{lem}\label{lem:F(X)}

Let $\X\in\CX_n$ and let $\mathrm{F}(\X) := \langle X; \lePY, g, \rank, \Tp\rangle$ be the structure defined above. Then
\begin{enumerate}[label=\textup{(B\arabic*)},leftmargin=1.7\parindent,itemsep=0pt] 

\item\label{B1} $\langle X; \lePY, \Tp\rangle$ is a Priestley space,

\item\label{B2} $g$ is a continuous retraction,

\item\label{B3} $x\lePY y$ implies $g(x) = g(y)$, for all $x,y\in X$,

\item\label{B4} $g(X)$ is a union of order components of $\langle X; \lePY \rangle$,

\item\label{B5} $\rank \colon X \to [0, n]$ is an $n$-ranking of $\langle X; \lePY, \Tp\rangle$,

\item\label{B6} $g(X) = \{\,x\in X \mid \rank(x) = 0\,\}$.

\end{enumerate}

\end{lem}

\begin{proof}
\ref{B1} follows from Lemma~\ref{lem:lePYonX}(c).

\ref{B2} follows immediately from the fact that $g_k$ maps $X_k$ to $X_0$, for all $k\in [1, n]$, and the fact that $g_0 = \id_{X_0}$.

\ref{B3} is immediate from the definitions of $g$ and $\lePY$
and the fact that $\X$ satisfies \ref{A2} and~\ref{A3}.

\ref{B4} follows from the fact that $g(X) = X_0$ and the fact that the 
relation $\lejk$ between distinct sorts applies only for $j < k$ in $[1, n]$.

Since $\{ X_0, \dots, X_n\}$ is a partition (with possibly empty blocks) 
of $X$ into closed subsets, the function
$\rank\colon X\to [0, n]$ is continuous. To prove~\ref{B5} it remains to prove that $\rank$ is order-preserving. Let $x, y\in X$ with $x \lePY y$. Then either $x, y \in X_k$ with $x\lek y$, for some $k\in [0, n]$, in which case $\rank(x) = k = \rank(y)$, or $x\in X_j$ and $y\in X_k$ with $x\lejk y$, for some $j < k$ in $[1, n]$, in which case $\rank(x) = j < k = \rank(y)$. Hence $\rank$ is order-preserving.

Finally, $g(X) = X_0 = \{\,x\in X \mid \rank(x) = 0\,\}$ by the definitions
of $g$ and of $\rank$, whence~\ref{B6} holds.
\end{proof}

\begin{df}\label{df:CYn}
We now define $\CY_n$ to be the category whose objects are topological structures $\langle Y; \lePY, g, \rank, \Tp\rangle$ satisfying~\ref{B1}--\ref{B6} and whose morphisms are continuous maps that preserve $\lePY$, $g$ and $\rank$. The topological structures in $\CY_n$ can be made first order by removing the `operation' $\rank$, adding $n+1$ topologically closed unary relations 
$Y_0, \dots, Y_n$, and replacing the assumption that $\rank$ is order-preserving by the assumption that $Y_0, \dots, Y_n$ form a partition of $Y$ (with possibly empty blocks) along with the following axiom for each $k\in [0, n]$:
\[
(\forall x, y\in Y)
\ x \in Y_k \And x \lePY y \implies y\in Y_k \text{ or } y\in Y_{k+1} \text{ or } \cdots \text{ or } y\in Y_n.
\]
\end{df}

\begin{df}
Let  $\Y = \langle Y; \lePY, g, \rank, \Tp\rangle$ be an object in $\CY_n$. We define a new structure
\[
\mathrm{G}(\Y)
:= \langle X_0 \du \cdots \du X_n  ; \mathcal G_{(n)}, \mathcal {S}_{(n)}, \Tp\rangle
\]
in the signature of $\M_n$ as follows:
\begin{itemize}[itemsep=0pt] 

\item $X_k := \{\, x\in Y \mid \rank(x) = k\,\}$, for all $k\in [0, n]$,

\item $g_k\colon X_k \to X_0$ is given by $g_k := g\rest{X_k}$, for all $k\in [1, n]$,

\item ${\lek} := {\lePY}\cap  {(X_k \times X_k)}$, for all $k\in [0, n]$,

\item ${\lejk} := {\lePY}\cap  {(X_j \times X_k)}$, for all $j < k$ in 
$[1, n]$.

\end{itemize}
\end{df}

We extend both $\mathrm{F}$ and $\mathrm{G}$ to morphisms by defining $\mathrm{F}(\varphi) = \varphi$ and $\mathrm{G}(\varphi) = \varphi$.

\goodbreak

\begin{thm}\label{thm:catequiv}
Fix $n \in \omega\comp\{0\}$. Then $\mathrm{F}\colon \CX_n \to \CY_n$ and $\mathrm{G}\colon \CY_n \to \CX_n$ are well-defined, mutually inverse category isomorphisms.
\end{thm}

\begin{proof}
By Lemma~\ref{lem:F(X)}, $\mathrm{F}$ is well defined at the object level. Assume that $\varphi\colon \X \to \Y$ is an $\CX_n$-morphism, for some $\X, \Y\in \CX_n$. The definitions of $g$ and $\lePY$ guarantee that they are preserved by~$\varphi$.
Since $\rank(x) = k$ if and only if $x\in X_k$ and since $\varphi$ preserves the sorts, it follows that $\varphi$ preserves $\rank$. Thus $\varphi\colon \mathrm{F}(\X) \to \mathrm{F}(\Y)$ is a $\CY_n$-morphism. Hence $\mathrm{F}$ is well defined and it is trivial that $\mathrm{F}$ is a functor.

Before proving that $\mathrm{G}$ is well defined we note
that it is almost trivial that $\mathrm{G}(\mathrm{F}(\X)) = \X$, for all $\X\in \CX_n$, and the only non-trivial part of proving that $\mathrm{F}(\mathrm{G}(\Y)) = \Y$, for all $\Y\in \CY_n$, is showing that the order $\lePY$ defined from $\mathrm{G}(\Y)$ agrees with the original order $\lePY$ on~$\Y$. To this end, consider an object $\Y = \langle Y; \lePY, g, \rank, \Tp\rangle$ in $\CY_n$ and denote the order defined on $\mathrm{F}(\mathrm{G}(\Y))$ from $\mathrm{G}(\Y)$ by~$\lePY'$. We must prove that ${\lePY} = {\lePY'}$. First, assume that $x,y\in Y$ with $x\lePY y$. Then $\rank(x)\le\rank(y)$, as $\Y$ satisfies~\ref{B5}.
\begin{itemize}[itemsep=0pt] 

\item If $\rank(x) = k = \rank(y)$, then in $\mathrm{G}(\Y)$ we have $x, y\in X_k$ and $x \lek y$.

\item If $\rank(x) = j < k = \rank(y)$, then in $\mathrm{G}(\Y)$ we have $x\in X_j$, $y\in X_k$ and $x \lejk y$.
\end{itemize}
In both cases, it follows from the definition of $\lePY'$ on $\mathrm{F}(\mathrm{G}(\Y))$ that $x\lePY' y$. Hence ${\lePY} \subseteq {\lePY'}$. Now let $x, y\in Y$ with $x\lePY' y$. From the definitions of $X_k$, $\lek$ and $\lejk$ in $\mathrm{G}(\Y)$ we have
\begin{itemize}[itemsep=0pt] 

\item $\rank(x) = k = \rank(y)$, for some $k\in [0, n]$ and $x\lePY y$ in $\Y$, or

\item $\rank(x) = j < k = \rank(y)$, for some $j, k\in [1, n]$ and $x\lePY y$ in $\Y$.

\end{itemize}
In both cases, we have $x\lePY y$ and hence ${\lePY'} \subseteq {\lePY}$. Thus ${\lePY'} = {\lePY}$, as required.

To prove that $\mathrm{G}$ is well defined at the object level we must prove that, for all $\Y\in \CY_n$, the structure $\mathrm{G}(\Y)$ satisfies \ref{A1}--\ref{A7}.
Let $\Y\in \CY_n$. Then~\ref{A1} follows from~\ref{B2}, and both~\ref{A2} and~\ref{A3} follow from~\ref{B3}. By Lemma~\ref{lem:lePYonX}, the topological structure $\mathrm{G}(\Y)$ satisfies~\ref{A4}--\ref{A7} if and only if the ordered space $\langle Y; \lePY', \Tp\rangle$ \emph{defined from the structure $\mathrm{G}(\Y)$} is a Priestley space. As $\mathrm{F}(\mathrm{G}(\Y)) = \Y$, we have $\langle Y; \lePY', \Tp\rangle = \langle Y; \lePY, \Tp\rangle$, which is a Priestley space as $\Y$ satisfies~\ref{B1}. Hence $\mathrm{G}(\Y)$ satisfies~\ref{A4}--\ref{A7}.

Finally, assume that $\varphi\colon \X\to \Y$ is a $\CY_n$-morphism.
It is almost trivial from the definitions that $\varphi\colon \mathrm{G}(\X)\to \mathrm{G}(\Y)$ is an $\CX_n$-morphism.
This completes the proof.
\end{proof}

\section{The Priestley dual of the $\CCD$-reduct of an algebra in $\CV_n$}\label{sec:Prdual}

Priestley duality establishes a dual category equivalence between
the variety $\CCD$ of bounded distributive lattices, \emph{qua}
category,  and the category $\CP$ of Priestley spaces with
continuous order-preserving maps as morphisms (see
Priestley~\cite{Pri70, Pri72} and Davey and Priestley~\cite{ILO}).
The duality is given by natural hom-functors. Let $\two =\langle
\{0, 1\}; \wedge, \vee, 0, 1\rangle$ be the two-element bounded
lattice and let $\mathbbm 2 = \langle \{0, 1\}; \le, \Tp\rangle$ be
the two-element chain with the discrete topology. Then, at the
object level:

\begin{itemize}[itemsep=0pt]

\item the dual of $\A\in \CCD$ is $\mathrm H(\A) = \CCD(\A, \two)$, with its order and topology inherited from the power~$\mathbbm 2^A$;

\item the dual of $\X\in \CP$ is $\mathrm K(\X) = \CP(\X, \mathbbm 2)$, a sublattice of~$\two ^X$.

\end{itemize}

\goodbreak 

Once again, throughout this section, we fix $n \in \omega\comp\{0\}$. We now have functors:

\begin{itemize}[itemsep=0pt] 

\item $\mathrm{D}\colon \CV_n \to \CX_n$, mapping each algebra $\A\in \CV_n$ to its natural (multi-sorted) dual $\mathrm{D}(\A)$, and

\item $\mathrm{H}'\colon \CV_n \to \CP$, mapping each algebra $\A\in \CV_n$ to the Priestley dual $\mathrm{H}'(\A) := \mathrm{H}(\A^\flat)$, where
$\A^\flat = \langle A; \wedge, \vee, \mbf_0, \mbt_0\rangle$
is the bounded-distributive-lattice reduct of $\A$.

\end{itemize}
To be able to use the natural dual $\mathrm{D}(\A)$ and the Priestley dual $\mathrm{H}'(\A)$ in tandem, we wish to give an explicit description of the functor ${\mathrm{P}\colon \CX_n \to \CP}$, mapping each (multi-sorted) structure $\X\in \CX_n$ to
\[
\mathrm{P}(\X) := \mathrm{H}'(\mathrm{E}(\X)) = \mathrm{H}(\mathrm{E}(\X)^\flat).
\]

As $\CX_n$ is isomorphic to $\CY_n$ and the objects in $\CY_n$ are Priestley spaces with additional structure, it is natural to define $\mathrm{P}(\X)$ in terms of $\mathrm{F}(\X)$: we shall define $\mathrm{P}(\X)$ to be the disjoint union of $\mathrm{F}(\X)$ and its order-theoretic dual $\mathrm{F}(\X)^\partial$ with additional ordering defined via the map~$g$.

\begin{df}\label{def:P(X)}
Let  $\X = \langle X_0 \du \cdots \du X_n  ; \mathcal G_{(n)}, \mathcal {S}_{(n)}, \Tp\rangle$ be an object in $\CX_n$, define $X :=  X_0
\du \cdots \du X_n$ and let $\mathrm{F}(\X) = \langle X; g, \le, \rank, \Tp\rangle$ be the corresponding object in~$\CY_n$. Recall that $g$ is a retraction of $X$ onto $X_0$. We define
\[
\mathrm{P}(\X) := \langle X\du \widehat X; \leP, \Tp\rangle,
\]
where $\widehat X := \{\widehat x \mid x \in X\}$, $\Tp$ is the disjoint union topology, and $\leP$ is defined on $X\du \widehat X$ by
\begin{enumerate}[label=\textup{(\arabic*)},itemsep=0pt] 

\item for $x, y \in X$: \ $x\leP y \iff  x\le y$,

\item for $\widehat x, \widehat y \in \widehat X$: \ $\widehat x\leP \widehat y \iff  x\ge y$,

\item for $x \in X\comp X_0$ and $y\in X_0$: \   $x\leP y \iff g(x) \le y$,

\item for $x \in X\comp X_0$ and $\widehat y\in \widehat X_0$: \   $x\leP \widehat y \iff g(x) \ge y$,

\item for $x\in X_0$ and $\widehat y \in \widehat X\comp \widehat X_0$: \   $x\leP \widehat y \iff x \le g(y)$,

\item for $\widehat x\in \widehat X_0$ and $\widehat y \in \widehat X\comp \widehat X_0$: \   $\widehat x\leP \widehat y \iff x \ge g(y)$,

\item for $x\in X\comp X_0$ and $\widehat y \in \widehat X\comp \widehat X_0$: \ $x\leP \widehat y \iff g(x) \le g(y) \text{ or } g(x)\ge g(y)$.

\end{enumerate}

Part (7) of the definition of $\leP$ can be thought of as obtaining $x\leP \widehat y$ by passing through $X_0$ (via $g(x) \le g(y)$) or through $\widehat X_0$ (via $g(x) \ge g(y)$). For example, if $x\in X\comp X_0$ and $\widehat y \in \widehat X\comp \widehat X_0$ with $g(x)\ge g(y)$, then we have
\[
g(x) \ge g(x) \And g(x)\ge g(y) \And g(y) \ge g(y).
\]
Since $g(x), g(y) \in X_0$, by applying, in order, (4) then (2) then (6), we have
\[
x\leP \widehat{g(x)} \And \widehat{g(x)} \leP \widehat{g(y)} \And \widehat{g(y)}\leP \widehat y,
\]
and therefore, if $\leP$ is to be transitive, we must have $x\leP \widehat y$.

We extend $\mathrm P$ to morphisms in the obvious way: if $\varphi\colon \X \to \Y$ is an $\CX_n$-morphism, then $\mathrm P(\varphi) \colon \mathrm P(\X) \to \mathrm P(\Y)$ is given by $\mathrm P(\varphi)(x) = \varphi(x)$, for $x\in X$, and $\mathrm P(\varphi)(\widehat x) = \widehat{\varphi(x)}$, for $\widehat x\in \widehat X$.

\end{df}

The next section will be devoted to proving the following result.

\begin{thm}\label{thm:trans}
Fix $n \in \omega\comp\{0\}$. Then $\mathrm P\colon \CX_n \to \CP$ is a well-defined functor.
For each $\A\in \CV_n$, let $\A^\flat = \langle A; \wedge, \vee, \mbf_0,
\mbt_0\rangle$ be its bounded-distributive-lattice reduct. Then $\mathrm H(\A^\flat)$ is isomorphic to $\mathrm P(\mathrm D(\A))$.
\end{thm}

Figure~\ref{fig:HFVn1} shows the ordered set 
$\mathrm P(\MT_n) \cong \mathrm P(\mathrm D(\F_{\CV_n}(1))) \cong \mathrm H(\F_{\CV_n}(1)^\flat)$ 
(cf.~Figure~\ref{fig:FofMn}).

\begin{figure}[ht]
\centering
\small
\begin{tikzpicture}[scale=1.25] 

\begin{scope}[scale=0.5,xshift=1cm]
  \node[unshaded] (topa) at (0,1.4) {};
  \node[unshaded] (fa) at (-1,2.4) {};
  \node[unshaded] (ta) at (1,2.4) {};
  \node[unshaded] (bota) at (0,3.4) {};
  \draw[order] (topa) -- (fa) -- (bota);
  \draw[order] (topa) -- (ta) -- (bota);

  \node[label,anchor=west,xshift=4pt,yshift=1pt] at (ta) {$\widehat{\mbt^0}$};
  \node[label,anchor=east,yshift=-3pt,xshift=2pt] at (topa) {$\widehat{\top^0}$};
  \node[label,anchor=east,yshift=2pt,xshift=2pt] at (bota) {$\widehat{\bot^0}$};
  \node[label,anchor=east] at (fa) {$\widehat{\mbf^0}$};
\end{scope}

\begin{scope}[scale=0.5, xshift=6cm]
  \node[unshaded] (botb) at (0,1.4) {};
  \node[unshaded] (fb) at (-1,2.4) {};
  \node[unshaded] (tb) at (1,2.4) {};
  \node[unshaded] (topb) at (0,3.4) {};
  \draw[order] (botb) -- (fb) -- (topb);
  \draw[order] (botb) -- (tb) -- (topb);

  \node[label,anchor=west] at (tb) {$\mbt^{0}$};
  \node[label,anchor=west,yshift=-4pt,xshift=-4pt] at (botb) {$\bot^0$};
  \node[label,anchor=west,yshift=2pt,xshift=-4pt] at (topb) {$\top^0$};
  \node[label,anchor=east,xshift=-1pt,yshift=-1pt] at (fb) {$\mbf^{0}$};
\end{scope}

\begin{scope}[scale=0.5, xshift=3cm, yshift=-4cm]

  \node[unshaded] (topg1) at (0,1){};
  \node[unshaded] (topg2) at (0,-1.5){};
  \node[unshaded] (topg3) at (0,-0.5){};

  \node at ($0.35*(topg1)+0.65*(topg3)$) {$\cdot$};
  \node at ($0.5*(topg1)+0.5*(topg3)$) {$\cdot$};
  \node at ($0.65*(topg1)+0.35*(topg3)$) {$\cdot$};
  \draw (topg1) -- ($0.8*(topg1)+0.2*(topg3)$);
  \draw (topg3) -- ($0.2*(topg1)+0.8*(topg3)$);
  \draw[order] (topg1) -- (topa);
  \draw[order] (topg1) -- (topb);
  \draw[order] (topg2) -- (topg3);

  \node[label,anchor=east] at (topg1) {$\top^{n}$};
  \node[label,anchor=east] at (topg2) {$\top^1$};
  \node[label,anchor=east] at (topg3) {$\top^2$};
\end{scope}

\begin{scope}[scale=0.5, xshift=4cm, yshift=-4cm]

  \node[unshaded] (botg1) at (0,1){};
  \node[unshaded] (botg2) at (0,-1.5){};
  \node[unshaded] (botg3) at (0,-0.5){};

  \node at ($0.35*(botg1)+0.65*(botg3)$) {$\cdot$};
  \node at ($0.5*(botg1)+0.5*(botg3)$) {$\cdot$};
  \node at ($0.65*(botg1)+0.35*(botg3)$) {$\cdot$};
  \draw (botg1) -- ($0.8*(botg1)+0.2*(botg3)$);
  \draw (botg3) -- ($0.2*(botg1)+0.8*(botg3)$);
  \draw[order] (botg1) -- (botb);
  \draw[order] (botg1) -- (bota);
  \draw[order] (botg2) -- (botg3);

  \node[label,anchor=west] at (botg1) {$\bot^{n}$};
  \node[label,anchor=west] at (botg2) {$\bot^1$};
  \node[label,anchor=west] at (botg3) {$\bot^2$};
\end{scope}

\begin{scope}[scale=0.5, xshift=3cm, yshift=8cm]

  \node[unshaded] (botd1) at (0,0) {};
  \node[unshaded] (botd2) at (0,2.5){};
  \node[unshaded] (botd3) at (0,1.5){};

  \node at ($0.35*(botd1)+0.65*(botd3)$) {$\cdot$};
  \node at ($0.5*(botd1)+0.5*(botd3)$) {$\cdot$};
  \node at ($0.65*(botd1)+0.35*(botd3)$) {$\cdot$};
  \draw (botd1) -- ($0.8*(botd1)+0.2*(botd3)$);
  \draw (botd3) -- ($0.2*(botd1)+0.8*(botd3)$);
 \draw[order] (botd1) -- (bota);
  \draw[order] (botd1) -- (botb);
  \draw[order] (botd2) -- (botd3);

  \node[label,anchor=east] at (botd1) {$\widehat{\bot^{n}}$};
  \node[label,anchor=east] at (botd2) {$\widehat{\bot^1}$};
  \node[label,anchor=east] at (botd3) {$\widehat{\bot^2}$};
\end{scope}

\begin{scope}[scale=0.5, xshift=4cm, yshift=8cm]

  \node[unshaded] (topd1) at (0,0) {};
  \node[unshaded] (topd2) at (0,2.5){};
  \node[unshaded] (topd3) at (0,1.5){};

  \node at ($0.35*(topd1)+0.65*(topd3)$) {$\cdot$};
  \node at ($0.5*(topd1)+0.5*(topd3)$) {$\cdot$};
  \node at ($0.65*(topd1)+0.35*(topd3)$) {$\cdot$};
  \draw (topd1) -- ($0.8*(topd1)+0.2*(topd3)$);
  \draw (topd3) -- ($0.2*(topd1)+0.8*(topd3)$);
  \draw[order] (topd1) -- (topa);
  \draw[order] (topd1) -- (topb);
  \draw[order] (topd2) -- (topd3);

  \node[label,anchor=west] at (topd1) {$\widehat{\top^{n}}$};
  \node[label,anchor=west] at (topd2) {$\widehat{\top^1}$};
  \node[label,anchor=west] at (topd3) {$\widehat{\top^2}$};
\end{scope}

\begin{scope}[scale=0.5, xshift=-1cm, yshift=-3cm]

  \node[unshaded] (fg1) at (0,0) {};
  \node[unshaded] (0g1) at (0,1){};
  \node[unshaded] (0gz) at (-2.5,-1.5) {};
  \node[unshaded] (fgz) at (-2.5,-2.5){};
  \node[unshaded] (0gzz) at (-1.5,-0.5) {};
  \node[unshaded] (fgzz) at (-1.5,-1.5){};

  \draw[order] (fg1) -- (0g1);
  \draw[order] (fgz) -- (0gz);
  \draw[order] (fgz) -- (fgzz);
  \draw[order] (0gz) -- (0gzz);
  \draw[order] (fgzz) -- (0gzz);
  \node at ($0.35*(fg1)+0.65*(fgzz)$) {$\cdot$};
  \node at ($0.5*(fg1)+0.5*(fgzz)$) {$\cdot$};
  \node at ($0.65*(fg1)+0.35*(fgzz)$) {$\cdot$};
  \node at ($0.35*(0g1)+0.65*(0gzz)$) {$\cdot$};
  \node at ($0.5*(0g1)+0.5*(0gzz)$) {$\cdot$};
  \node at ($0.65*(0g1)+0.35*(0gzz)$) {$\cdot$};
  \draw (fg1) -- ($0.8*(fg1)+0.2*(fgzz)$);
  \draw (0g1) -- ($0.8*(0g1)+0.2*(0gzz)$);
  \draw (fgzz) -- ($0.2*(fg1)+0.8*(fgzz)$);
  \draw (0gzz) -- ($0.2*(0g1)+0.8*(0gzz)$);
  \draw[order] (0g1) -- (fa);
  \draw[order] (0g1) -- (fb);

  \node[label,anchor=east] at (0g1) {$\zero^{n}$};
  \node[label,anchor=west] at (fg1) {$\mbf^{n}$};
  \node[label,anchor=east] at (0gz) {$\zero^1$};
  \node[label,anchor=west] at (fgz) {$\mbf^1$};
  \node[label,anchor=east] at (0gzz) {$\zero^2$};
  \node[label,anchor=west] at (fgzz) {$\mbf^2$};

  \end{scope}

\begin{scope}[scale=0.5, xshift=8cm, yshift=-3cm]

  \node[unshaded] (tg1) at (0,0) {};
  \node[unshaded] (1g1) at (0,1){};
  \node[unshaded] (1gz) at (2.5,-1.5) {};
  \node[unshaded] (tgz) at (2.5,-2.5){};
  \node[unshaded] (1gzz) at (1.5,-0.5) {};
  \node[unshaded] (tgzz) at (1.5,-1.5){};

  \draw[order] (tg1) -- (1g1);
  \draw[order] (1g1) -- (ta);
  \draw[order] (1g1) -- (tb);
  \draw[order] (1gz) -- (1gzz);
  \draw[order] (tgz) -- (tgzz);
  \draw[order] (tgzz) -- (1gzz);
  \node at ($0.35*(tg1)+0.65*(tgzz)$) {$\cdot$};
  \node at ($0.5*(tg1)+0.5*(tgzz)$) {$\cdot$};
  \node at ($0.65*(tg1)+0.35*(tgzz)$) {$\cdot$};
  \node at ($0.35*(1g1)+0.65*(1gzz)$) {$\cdot$};
  \node at ($0.5*(1g1)+0.5*(1gzz)$) {$\cdot$};
  \node at ($0.65*(1g1)+0.35*(1gzz)$) {$\cdot$};
  \draw (tg1) -- ($0.8*(tg1)+0.2*(tgzz)$);
  \draw (1g1) -- ($0.8*(1g1)+0.2*(1gzz)$);
  \draw (tgzz) -- ($0.2*(tg1)+0.8*(tgzz)$);
  \draw (1gzz) -- ($0.2*(1g1)+0.8*(1gzz)$);
  \draw (tg1) -- ($0.8*(tg1)+0.2*(tgzz)$);
  \draw (1g1) -- ($0.8*(1g1)+0.2*(1gzz)$);
  \draw (tgzz) -- ($0.2*(tg1)+0.8*(tgzz)$);
  \draw (1gzz) -- ($0.2*(1g1)+0.8*(1gzz)$);

  \node[label,anchor=west] at (1g1) {$\one^{n}$};
  \node[label,anchor=east] at (tg1) {$\mbt^{n}$};
  \node[label,anchor=west] at (1gz) {$\one^1$};
  \node[label,anchor=east] at (tgz) {$\mbt^1$};
  \node[label,anchor=west] at (1gzz) {$\one^2$};
  \node[label,anchor=east] at (tgzz) {$\mbt^2$};

  \end{scope}

\begin{scope}[scale=0.5, xshift=-1cm, yshift=7cm]

  \node[unshaded] (0d1) at (0,0) {};
  \node[unshaded] (fd1) at (0,1){};
  \node[unshaded] (fd2) at (-2.5,3.5){};
  \node[unshaded] (0d2) at (-2.5,2.5){};
  \node[unshaded] (fd22) at (-1.5,2.5){};
  \node[unshaded] (0d22) at (-1.5,1.5){};

  \draw[order] (fd1) -- (0d1);
  \draw[order] (tgz) -- (1gz);
  \draw[order] (fd2) -- (fd22);
  \draw[order] (0d2) -- (0d22);
  \draw[order] (fd22) -- (0d22);
  \node at ($0.35*(fd1)+0.65*(fd22)$) {$\cdot$};
  \node at ($0.5*(fd1)+0.5*(fd22)$) {$\cdot$};
  \node at ($0.65*(fd1)+0.35*(fd22)$) {$\cdot$};
  \node at ($0.35*(0d1)+0.65*(0d22)$) {$\cdot$};
  \node at ($0.5*(0d1)+0.5*(0d22)$) {$\cdot$};
  \node at ($0.65*(0d1)+0.35*(0d22)$) {$\cdot$};
  \draw (fd1) -- ($0.8*(fd1)+0.2*(fd22)$);
  \draw (0d1) -- ($0.8*(0d1)+0.2*(0d22)$);
  \draw (fd22) -- ($0.2*(fd1)+0.8*(fd22)$);
  \draw (0d22) -- ($0.2*(0d1)+0.8*(0d22)$);
  \draw[order] (0d1) -- (fa);
  \draw[order] (0d1) -- (fb);

  \node[label,anchor=east] at (0d1) {$\widehat{\zero^{n}}$};
  \node[label,anchor=west] at (fd1) {$\widehat{\mbf^{n}}$};
  \node[label,anchor=east] at (0d2) {$\widehat {\zero^1}$};
  \node[label,anchor=west,xshift=3pt,yshift=1pt] at (fd2) {$\widehat{\mbf^1}$};
  \node[label,anchor=east] at (0d22) {$\widehat {\zero^2}$};
  \node[label,anchor=west,xshift=3pt,yshift=1pt] at (fd22) {$\widehat{\mbf^2}$};

  \end{scope}

\begin{scope}[scale=0.5, xshift=8cm, yshift=7cm]

  \node[unshaded] (1d1) at (0,0) {};
  \node[unshaded] (td1) at (0,1){};
  \node[unshaded] (1dz) at (2.5,2.5){};
  \node[unshaded] (tdz) at (2.5,3.5){};
  \node[unshaded] (1dzz) at (1.5,1.5){};
  \node[unshaded] (tdzz) at (1.5,2.5){};

  \draw[order] (td1) -- (1d1);
  \draw[order] (0d2) -- (fd2);
  \draw[order] (1dz) -- (tdz);
  \draw[order] (1dz) -- (1dzz);
  \draw[order] (tdz) -- (tdzz);
  \draw[order] (tdzz) -- (1dzz);
  \node at ($0.35*(td1)+0.65*(tdzz)$) {$\cdot$};
  \node at ($0.5*(td1)+0.5*(tdzz)$) {$\cdot$};
  \node at ($0.65*(td1)+0.35*(tdzz)$) {$\cdot$};
  \node at ($0.35*(1d1)+0.65*(1dzz)$) {$\cdot$};
  \node at ($0.5*(1d1)+0.5*(1dzz)$) {$\cdot$};
  \node at ($0.65*(1d1)+0.35*(1dzz)$) {$\cdot$};
  \draw (td1) -- ($0.8*(td1)+0.2*(tdzz)$);
  \draw (1d1) -- ($0.8*(1d1)+0.2*(1dzz)$);
  \draw (tdzz) -- ($0.2*(td1)+0.8*(tdzz)$);
  \draw (1dzz) -- ($0.2*(1d1)+0.8*(1dzz)$);
  \draw (td1) -- ($0.8*(td1)+0.2*(tdzz)$);
  \draw (1d1) -- ($0.8*(1d1)+0.2*(1dzz)$);
  \draw (tdzz) -- ($0.2*(td1)+0.8*(tdzz)$);
  \draw (1dzz) -- ($0.2*(1d1)+0.8*(1dzz)$);
  \draw[order] (1d1) -- (ta);
  \draw[order] (1d1) -- (tb);

  \node[label,anchor=west] at (1d1) {$\widehat{\one^{n}}$};
  \node[label,anchor=east] at (td1) {$\widehat{\mbt^{n}}$};
  \node[label,anchor=west] at (1dz) {$\widehat{\one^1}$};
  \node[label,anchor=east] at (tdz) {$\widehat{\mbt^1}$};
  \node[label,anchor=west] at (1dzz) {$\widehat{\one^2}$};
  \node[label,anchor=east] at (tdzz) {$\widehat{\mbt^2}$};

  \end{scope}

\end{tikzpicture}
\caption{The ordered set $\mathrm P(\MT_n)  
\cong
\mathrm H(\F_{\CV_n}(1)^\flat)$.}\label{fig:HFVn1}
\end{figure}

As an application of Theorem~\ref{thm:trans}, we will calculate the cardinality of the free algebra $\F_{\CV_n}(1)$, for all $n\in \omega\comp\{0\}$.
Throughout the calculation we will use the arithmetic of down-sets without specific reference (see Proposition~1.31, Exercise~1.18, Lemma~5.18 and Exercise~5.19 in Davey and Priestley~\cite{ILO}).

\begin{thm}\label{thm:sizefree}
Let $n\in \omega\comp\{0\}$. Then
\[
|F_{\CV_n}(1)| = \tfrac12\big(n^6 + 10n^5 + 42n^4 + 102n^3 + 157n^2 + 148n + 72\big).
\]
\end{thm}
\begin{proof}
Since $\F_{\CV_n}(1)^\flat \cong \mathrm K\mathrm H(\F_{\CV_n}(1)^\flat) \cong \mathrm K(\mathrm P(\MT_n))$
is isomorphic to the lattice $\boldsymbol{\mathcal O}(\mathrm P(\MT_n))$ of down-sets of $\mathrm P(\MT_n)$, we need to count the number of down-sets of the ordered set shown in Figure~\ref{fig:HFVn1}. Divide $\mathrm P(\MT_n)$ into three subsets:
\begin{itemize}[itemsep=0pt] 

\item the bottom $\B$, isomorphic to $(\two \times \mathbf n) \du \mathbf n \du \mathbf n \du (\two \times \mathbf n)$,

\item the centre $\C$, isomorphic to $\two^2 \du \two^2$, and

\item the top $\T$, also isomorphic to $(\two \times \mathbf n) \du \mathbf n \du \mathbf n \du (\two \times \mathbf n)$.

\end{itemize}
We shall count the down-sets of $\mathrm P(\MT_n)$ by first counting the number that do not intersect the top $T$ and then counting the number that do.

Since the ordered set $\two \times \mathbf n$ occurs four times within $\mathrm P(\MT_n)$, it will be useful to know that $|\mathcal O(\two \times \mathbf n)| = \frac12(n +1)(n + 2)$. This is an easy calculation; indeed, $\boldsymbol{\mathcal O}(\two \times \mathbf n)$, the coproduct in $\CCD$ of $\mathbf 3$ and $\mathbf{n \boldsymbol+ 1}$, is as shown in Figure~\ref{fig:O(2timesn)}.

\begin{figure}[ht]
\centering
\begin{tikzpicture}[scale=0.5]
  \begin{scope}
    \node[shaded] (0) at (0,-1.4) {};
    \node[unshaded] (v1) at (0,0) {};
    \node[shaded] (b1) at (1,1) {};
    \node[unshaded] (w1) at (-1,1) {};
    \node[unshaded] (w2) at (0,2) {};
    \node[shaded] (b2) at (1,3) {};
    \node[unshaded] (x1) at (-2,2) {};
    \node[unshaded] (x2) at (-1,3) {};
    \node[unshaded] (x3) at (0,4) {};
    \node[shaded] (bnm2) at (1,5) {};
    \node[unshaded] (y1) at (-3,3) {};
    \node[unshaded] (y2) at (-2,4) {};
    \node[unshaded] (y3) at (-1,5) {};
    \node[unshaded] (y4) at (0,6) {};
    \node[shaded] (bnm1) at (1,7) {};
    \node[shaded] (a) at (-4,4) {};
    \node[unshaded] (z2) at (-3,5) {};
    \node[unshaded] (z3) at (-2,6) {};
    \node[unshaded] (z4) at (-1,7) {};
    \node[unshaded] (z5) at (0,8) {};
    \node[shaded] (1) at (0,9.4) {};
    \draw[order] (0) -- (v1) -- (b1);
    \draw[order] (w1) -- (w2) -- (b2);
    \draw[order] (x1) -- (x2) -- (x3) -- ($(x3)+(6.5pt,6.5pt)$);
    \node at ($0.35*(x3)+0.65*(bnm2)$) {$\cdot$};
    \node at ($0.5*(x3)+0.5*(bnm2)$) {$\cdot$};
    \node at ($0.65*(x3)+0.35*(bnm2)$) {$\cdot$};
    \draw[order] ($(bnm2)+(-6.5pt,-6.5pt)$) -- (bnm2);
    \draw[order] (y1) -- (y2) -- (y3) -- ($(y3)+(6.5pt,6.5pt)$);
    \node at ($0.35*(y3)+0.65*(y4)$) {$\cdot$};
    \node at ($0.5*(y3)+0.5*(y4)$) {$\cdot$};
    \node at ($0.65*(y3)+0.35*(y4)$) {$\cdot$};
    \draw[order] ($(y4)+(-6.5pt,-6.5pt)$) -- (y4) -- (bnm1);
    \draw[order] (a) -- (z2) -- (z3) -- ($(z3)+(6.5pt,6.5pt)$);
    \node at ($0.35*(z3)+0.65*(z4)$) {$\cdot$};
    \node at ($0.5*(z3)+0.5*(z4)$) {$\cdot$};
    \node at ($0.65*(z3)+0.35*(z4)$) {$\cdot$};
    \draw[order] ($(z4)+(-6.5pt,-6.5pt)$) -- (z4) -- (z5) -- (1);
    \draw[order] (v1) -- (w1) -- ($(w1)+(-6.5pt,6.5pt)$);
    \node at ($0.35*(x1)+0.65*(w1)$) {$\cdot$};
    \node at ($0.5*(x1)+0.5*(w1)$) {$\cdot$};
    \node at ($0.65*(x1)+0.35*(w1)$) {$\cdot$};
    \draw[order] ($(x1)+(6.5pt,-6.5pt)$) -- (x1) -- (y1) -- (a);
    \draw[order] (b1) -- (w2) -- ($(w2)+(-6.5pt,6.5pt)$);
    \node at ($0.35*(x2)+0.65*(w2)$) {$\cdot$};
    \node at ($0.5*(x2)+0.5*(w2)$) {$\cdot$};
    \node at ($0.65*(x2)+0.35*(w2)$) {$\cdot$};
    \draw[order] ($(x2)+(6.5pt,-6.5pt)$) --  (x2) -- (y2) -- (z2);
    \draw[order] (b2) -- ($(b2)+(-6.5pt,6.5pt)$);
    \node at ($0.35*(x3)+0.65*(b2)$) {$\cdot$};
    \node at ($0.5*(x3)+0.5*(b2)$) {$\cdot$};
    \node at ($0.65*(x3)+0.35*(b2)$) {$\cdot$};
    \draw[order] ($(x3)+(6.5pt,-6.5pt)$) -- (x3) -- (y3) -- (z3);
    \draw[order] (bnm2) -- (y4) -- (z4);
    \draw[order] (bnm1) -- (z5);
    \node at ($0.39*(b2)+0.61*(bnm2)$) {$\vdots$};
    \node[label,anchor=west] at (0) {$0$};
    \node[label,anchor=west] at (b1) {$b_1$};
    \node[label,anchor=west] at (b2) {$b_2$};
    \node[label,anchor=west] at (bnm2) {$b_{n-2}$};
    \node[label,anchor=west] at (bnm1) {$b_{n-1}$};
    \node[label,anchor=west] at (1) {$1$};
    \node[label,anchor=east] at (a) {$a$};
    \node[label,anchor=west] at (4,4) {$\phantom{a}$};
  \end{scope}
\end{tikzpicture}
\caption{The lattice $\boldsymbol{\mathcal O}(\two \times \mathbf n)$}\label{fig:O(2timesn)}
\end{figure}
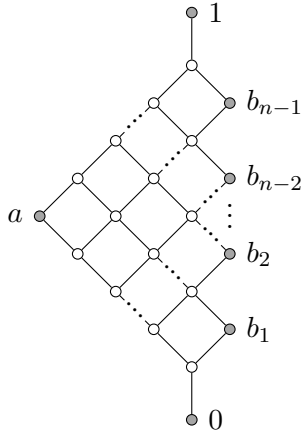

\begin{claim}
The number of down-sets of $\mathrm P(\MT_n)$ that do not intersect the top $T$ is
\[
f(n) = \tfrac14\big(n^6 + 10n^5 + 41n^4 + 96n^3 + 148n^2 + 148n + 144\big).
\]
\end{claim}
\begin{proof}[Proof of Claim 1]
Let $U$ be a down-set of $\mathrm P(\MT_n)$ that does not intersect~$T$. The intersection $U\cap C$ is one of the 36 down-sets of $\C$. The number of such $U$ for a given intersection $U \cap C$ is given in Table~\ref{tab:f(n)}. (To save space each superscript $0$ has been omitted in the first column of the table.) We will explain how the numbers are obtained in three typical cases.

\begin{table}[htp]
\caption{Calculating $f(n)$}
\begin{center}
\renewcommand{\arraystretch}{1.25}
\begin{tabular}{|l|l|}
\hline
$U\cap C$ &  \# of such $U$\\\hline\hline
$\varnothing$  & $\frac14(n+1)^4(n+2)^2$\\\hline
$\{\bot\}$ or $\{\widehat \top\}$  & \strut$\frac14(n+1)^3(n+2)^2$\\\hline
$\{\bot, \mbf\}$ or $\{\bot, \mbt\}$ or $\{\widehat\top, \widehat\mbf\,\}$ or $\{\widehat\top, \widehat\mbt\,\}$  &$\frac12(n+1)^2(n+2)$\\\hline
$\{\bot, \mbf, \mbt\}$ or $\{\widehat\top, \widehat\mbf, \widehat\mbt\,\}$  &$n+1$\\\hline
$\{\bot,\widehat \top\}$  & $\frac14(n+1)^2(n+2)^2$\\\hline
$\{\bot, \widehat\top, \widehat\mbf\,\}$ or $\{\bot, \widehat\top, \widehat\mbt\,\}$ or $\{\bot, \mbf, \widehat\top\,\}$  & \\
or $\{\bot, \mbt,\widehat\top\,\} $ or $\{\bot, \mbf, \widehat\top, \widehat\mbf\,\}$ or $\{\bot, \mbt,\widehat\top, \widehat\mbt\,\}$  & $\frac12(n+1)(n+2)$\\\hline
each of the remaining 20 possibilities&  1\\\hline
\end{tabular}
\end{center}
\label{tab:f(n)}
\end{table}%

\smallskip
\noindent\textbf{Case 1:} $U \cap C = \varnothing$.
In this case $U$ is a down-set of $\B \cong  (\two \times \mathbf n) \du \mathbf n \du \mathbf n \du (\two \times \mathbf n)$. Since
\[
\boldsymbol{\mathcal O}(\B) \cong \boldsymbol{\mathcal O}(\two \times \mathbf n)^2 \times (\mathbf{n \boldsymbol+ 1})^2,
\]
we conclude that the number of such down-sets is
\[
|\mathcal O(\B)| = \big(\tfrac12(n +1)(n + 2)\big)^2 \times (n+1)^2 = \tfrac14(n+1)^4(n+2)^2.
\]

\smallskip
\noindent\textbf{Case 2:} $U \cap C = \{\bot^0\}$. As $U$ is a down-set containing  $\bot^0$, it must also contain~$\down\bot^n$. Hence $U\comp\down\bot^n$ is an arbitrary down-set of $B' := B\comp\down\bot^n$. The ordered set $\B'$ is isomorphic to $(\two \times \mathbf n) \du \mathbf n \du (\two \times \mathbf n)$. Thus there are
\[
|\mathcal O(\B')| = \big(\tfrac12(n +1)(n + 2)\big)^2 \times (n+1) = \tfrac14(n+1)^3(n+2)^2
\]
choices for $U\comp \down \bot^n$ and hence for~$U$.

\goodbreak 

\smallskip
\noindent\textbf{Case 3:} $U \cap C = \{\bot^0, \mbf^0\}$. In this case, the down-set $U$ must contain $\down\bot^n$, as well as $\down\zero^n$. The only possibilities for the remainder of $U$ are the down-sets of $\down\top^n\cup \down \one^n$. Since $\down\top^n$ and $\down \one^n$ are disjoint, there are
\begin{equation*}
|\mathcal O(\down\top^n\cup \down \one^n)| = |\mathcal O(\down\top^n) \times \mathcal O(\down \one^n)|= \tfrac12(n +1)(n + 2) \times (n+1) = \tfrac12(n+1)^2(n+2)
\end{equation*}
such possibilities.

\smallskip

The other cases are argued similarly. Multiplying each term in the second column of Table~\ref{tab:f(n)} by the number of sets in the corresponding cell in the
first column and summing gives
\begin{multline*}
f(n) = \tfrac14(n+1)^4(n+2)^2 + 2\cdot \tfrac14(n+1)^3(n+2)^2\\ + 4\cdot \tfrac12(n+1)^2(n+2) + 2\cdot(n+1) + \tfrac14(n+1)^2(n+2)^2\\
+ 6\cdot \tfrac12(n+1)(n+2) + 20\cdot 1
\end{multline*}
which, after simplification, gives the formula in the claim.
\renewcommand{\qedsymbol}{$\diamond$}
\end{proof}

\begin{claim}
The number of down-sets of $\mathrm P(\MT_n)$ that intersect the top $T$ is
\[
g(n) = \tfrac14\big(n^6 + 10n^5 + 43n^4 + 108n^3 + 166n^2 + 148n\big).
\]
\end{claim}
\begin{proof}[Proof of Claim 2]
Let $U$ be a down-set of $\mathrm P(\MT_n)$ that intersects~$T$. The intersection of $U$ with the set $\min(T) = \{\widehat{\zero^n}, \widehat{\bot^n}, \widehat{\top^n}, \widehat{\one^n}\}$ of minimal elements of $T$ is one of the 15 non-empty subsets of $\min(T)$. The number of such $U$ for a given intersection $U \cap \min(T)$ is given in Table~\ref{tab:g(n)}. (Once again, to save space each superscript $n$ has been omitted in the first column of the table.) We now explain how these numbers are obtained.

\begin{table}[htp]
\caption{Calculating $g(n)$}
\begin{center}
\renewcommand{\arraystretch}{1.25}
\begin{tabular}{|l|l|}
\hline
$U \cap \min(T)$ &  \# of such $U$\\\hline\hline
$\{\widehat\zero\}$ or $\{\widehat\one\}$&  $\big(\frac12(n+1)(n+2) - 1\big)\big(\frac12(n+1)(n+2) +8\big)$\\\hline
$\{\widehat\bot\}$ or $\{\widehat \top\}$ &  $5n$\\\hline
$\{\widehat\zero, \widehat\one\}$ &  $4\big(\frac12(n+1)(n+2) - 1\big)^2$\\\hline
$\{\widehat\zero, \widehat\top\}$ or $\{\widehat\zero, \widehat\bot\}$  & \\
\quad or $\{\widehat\top, \widehat\one\}$ or $\{\widehat\bot, \widehat\one\}$  &$3n\big(\frac12(n+1)(n+2) - 1\big)$\\\hline
$\{\widehat\top, \widehat\bot\}$ &  $n^2$ \\\hline
$\{\widehat\bot, \widehat\top, \widehat\one\}$ or $\{\widehat\zero, \widehat\top, \widehat\bot\}$ & $n^2\big(\frac12(n+1)(n+2) - 1\big)$\\\hline
$\{\widehat\zero, \widehat\top,\widehat\one\}$ or $\{\widehat\zero, \widehat\bot,\widehat\one\}$ &  $2n\big(\frac12(n+1)(n+2) - 1\big)^2$\\\hline 
$\{\widehat\zero, \widehat\bot, \widehat\top,\widehat\one\}$ &  $n^2\big(\frac12(n+1)(n+2) - 1\big)^2$\\\hline
\end{tabular}
\end{center}
\label{tab:g(n)}
\end{table}%

The down-set $U$ must contain $\bigcup \{\down a \mid a \in  U \cap \min(T)\}$.  The only possibility for the remainder of $U$ is a set of the form
\[
\bigcup\{V_a \mid a \in U \cap \min(T)\} \cup W,\tag{$*$}
\]
where $V_a$ is a non-empty down-set of $\up a$ and $W$ is a down-set of the ordered set $\P$ with underlying set
\[
P:=(C \cup B)\comp \big(\bigcup \{\down a \mid a \in  U \cap \min(T)\}\big).
\]
Since the union in $(*)$ is disjoint, the number of possibilities for $U$ is
\[
\prod\{|\mathcal O(\up a)| - 1 : a \in  U \cap \min(T)\} \times |\mathcal O(\P)|.\tag{$**$}
\]
If $a\in \{\widehat{\zero^n}, \widehat{\one^n}\}$, then $\up a \cong \two \times \mathbf n$ whence $|\mathcal O(\up a)| = \tfrac12(n+1)(n+2)$, and if $a \in \{\widehat{\bot^n}, \widehat{\top^n}\}$, then $\up a \cong \mathbf n$ whence $|\mathcal O(\up a)| = n+1$. In each case, it remains to calculate $P$ and $|\mathcal O(\P)|$. We present the details in three of the 15 cases. 
(We denote the \emph{linear sum} of ordered sets $\P$ and $\Q$ by $\P \oplus \Q$.)

\smallskip
\noindent\textbf{Case 1:} $U \cap \min(T) = \{\widehat{\zero^n}\}$. In this case
\[
P = (C\cup B)\comp \down\widehat{\zero^n} =  \down\one^n \cup \{\widehat\bot^0, \widehat{\mbt^0}\} \cup \{\top^0, \mbt^0\}.
\]
Since $\P \cong (\two \times \mathbf n) \oplus (\two \du \two)$, the product $(**)$ becomes
\begin{align*}
&\Big(|\mathcal O(\two \times \mathbf n)| - 1\Big) \times |\mathcal O\big((\two \times \mathbf n) \oplus (\two \du \two)\big)|\\
 = {}&\Big(|\mathcal O(\two \times \mathbf n)| - 1\Big) \times \Big(|\mathcal O\big((\two \times \mathbf n)\big)| + |\mathcal O\big(\two \du \two\big)| - 1\Big)\\
 = {}&\big(\tfrac12(n+1)(n+2) - 1\big)\big(\tfrac12(n+1)(n+2) +8\big)
\end{align*}
as $\mathcal O\big(\two \du \two\big) \cong \mathcal O(\two) \times \mathcal O(\two) \cong \mathbf 3 \times \mathbf 3$.

\smallskip
\noindent\textbf{Case 2:} $U \cap \min(T) = \{\widehat{\bot^n}\}$. In this case
\[
P = (C\cup B)\comp \down\widehat{\bot^n} = \{\top^0, \mbf^0, \mbt^0\}.
\]
Since $\P \cong \overline \two \oplus \one$, where $\overline \two$ denotes a two-element antichain, the product $(**)$ becomes
\[
\big(|\mathcal O(\mathbf n)| - 1\big) \times |\mathcal O(\overline \two \oplus \one)| = n \times 5.
\]

\smallskip
\noindent\textbf{Case 3:} $U \cap \min(T) = \{\widehat\zero^n, \widehat\one^n\}$. In this case,
\[
P = (C\cup B)\comp (\down\widehat{\zero^n}\cup \down\widehat{\one^n}) = \{\widehat{\bot^0}, \top^0\}.
\]
Since $\P \cong \overline\two$, the product $(**)$ becomes
\[
\Big(|\mathcal O(\two \times \mathbf n)| - 1\Big)^2 \times |\mathcal O(\overline\two)| = \big(\tfrac12(n+1)(n+2) - 1\big)^2 \times 4.
\]

The other cases are argued similarly. Multiplying each term in the second column of Table~\ref{tab:g(n)} by the number of sets in the corresponding cell in the
first column and summing gives
\begin{multline*}
g(n) = 2\cdot\big(\tfrac12(n+1)(n+2) - 1\big)\big(\tfrac12(n+1)(n+2) +8\big) + 2\cdot 5n\\
 + 4\big(\tfrac12(n+1)(n+2) - 1\big)^2 + 4\cdot 3n\big(\tfrac12(n+1)(n+2) - 1\big)   + n^2\\
 + 2\cdot n^2\big(\tfrac12(n+1)(n+2) - 1\big) + 2\cdot 2n\big(\tfrac12(n+1)(n+2) - 1\big)^2\\
  + n^2\big(\tfrac12(n+1)(n+2) - 1\big)^2
\end{multline*}
which, after simplification, gives the formula in the claim.
\renewcommand{\qedsymbol}{$\diamond$}
\end{proof}

Finally, to verify the formula for $|F_{\CV_n}(1)|$ given in the statement of the theorem, simply calculate $f(n) + g(n)$.
\end{proof}

\section{Piggybacking and the proof of Theorem~\ref{thm:trans}}\label{sec:piggy}

Our proof of Theorem~\ref{thm:trans} employs a very useful tool from the theory of piggyback dualities. We shall not develop the theory in full generality, but only in the form determined by the duality for~$\CV_n = \ISP(\{\M_0, \M_1, \dots, \M_n\})$ given in Theorem~\ref{cor:bigmultiduality}. We follow the presentation given in~\cite[Chapter 7]{CD98}. For the history of piggyback dualities, we refer the reader to Davey and Werner~\cite{DW85, DW86}, where the single-sorted version was introduced, to Davey and Priestley~\cite{DP-multi} where the theory was expanded to the setting of multi-sorted dualities, and to Davey, Haviar and Priestley~\cite{DHP-revisited} for the extension from algebras to structures of the single-sorted case.

Before we can present the tool that we will be using, we need to introduce some notation and terminology. We make use of the fact that the truth lattice of each algebra $\A$ in $\CV_n$ has a reduct $\A^\flat$ in $\CCD$, the variety of bounded distributive lattices. As in the previous sections, we fix $n \in \omega\comp\{0\}$.

\begin{df}
 For all $j, k\in [0, n]$, let $\omega_1 \in \CCD(\M_j^\flat,\two)$ and
$\omega_2 \in \CCD(\M_k^\flat,\two)$.
The set
\[
 (\omega_1,\omega_2)^{-1}(\leqslant) =
\{\,(a,b) \in M_j \times M_k \mid \omega_1(a) \leqslant \omega_2(b)\}
\]
forms a bounded sublattice of $\M_j^\flat\times \M_k^\flat$. We will denote by $\Rsub{\omega_1\omega_2}$ the (possibly empty) set of subuniverses of $\M_j\times \M_k$ that are maximal with respect to being contained in $(\omega_1,\omega_2)^{-1}(\leqslant)$. The maps $\omega_1$ and $\omega_2$ are referred to as \emph{carriers} and the relations in the set $\Rsub{\omega_1\omega_2}$ are referred to as the \emph{piggyback relations on $\{\M_0, \dots, \M_n\}$ determined by $\omega_1, \omega_2$}.
\end{df}

We turn our attention to the carrier maps that we will be using.
Let $\delta_0, \gamma_0\colon \M_0^\flat\to \two$ and  $\gamma_k, \delta_k\colon \M_k^\flat\to \two$, for $k\in [1, n]$, be determined by
\begin{alignat*}{2}
\delta_0^{-1}(1) &= \{\bot^0, \mbt^0\}, \qquad &
\gamma_0^{-1}(1) &= \{\top^0, \mbt^0\}, \\
\gamma_k^{-1}(1) &= \{\one^k\}, \qquad &
\delta_k^{-1}(1) &= \{\mbf^k, \bot^k, \top^k, \mbt^k, \one^k\};
\end{alignat*}
see Figure~\ref{fig:carriers}.

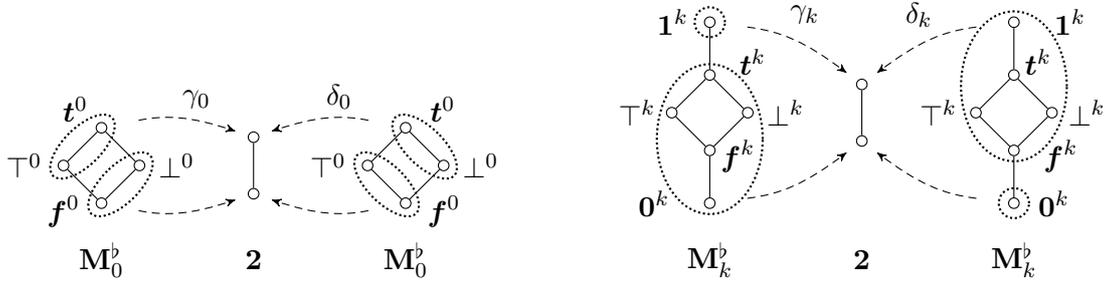
\begin{figure}[ht]
\centering 
{
\begin{tikzpicture}[scale=0.5]
\begin{scope}[xshift=-12cm]
  \node at (0,-1.5) {$\M_0^\flat$};
  \node[unshaded] (bot) at (0,0) {};
  \node[unshaded] (f) at (-1,1) {};
  \node[unshaded] (t) at (1,1) {};
  \node[unshaded] (top) at (0,2) {};
  \draw[order] (bot) -- (f) -- (top);
  \draw[order] (bot) -- (t) -- (top);
  \node[label,anchor=east,xshift=-2pt,yshift=-3pt] at (bot) {$\mbf^0$};
  \node[label,anchor=east,xshift=-2pt] at (f) {$\top^0$};
  \node[label,anchor=west,xshift=1pt] at (t) {$\bot^0$};
  \node[label,anchor=east,xshift=1pt,yshift=7pt] at (top) {$\mbt^0$};
  \node[thick,ellipse,draw,densely dotted,inner xsep=5pt,inner ysep=11pt,rotate=135] at ($0.5*(top)+0.5*(f)$) {};
  \node[thick,ellipse,draw,densely dotted,inner xsep=5pt,inner ysep=11pt,rotate=135] at ($0.5*(bot)+0.5*(t)$) {};
\end{scope}

\begin{scope}[xshift=-8cm]
  \node at (0,-1.5) {$\boldsymbol 2$};
  \node[unshaded] (0) at (0,0.25) {};
  \node[unshaded] (1) at (0,1.75) {};
  \draw[order] (0) -- (1);
  \node at (-1.5,2.75) {$\gamma_0$};
  \path (top) edge [curvy,bend left=15] node {} (1);
  \path (bot) edge [curvy,bend right=15] node {} (0);
\end{scope}

\begin{scope}[xshift=-4cm]
  \node at (0,-1.5) {$\M_0^\flat$};
  \node[unshaded] (bot) at (0,0) {};
  \node[unshaded] (f) at (-1,1) {};
  \node[unshaded] (t) at (1,1) {};
  \node[unshaded] (top) at (0,2) {};
  \draw[order] (bot) -- (f) -- (top);
  \draw[order] (bot) -- (t) -- (top);
  \node[label,anchor=west,xshift=2pt,yshift=-3pt] at (bot) {$\mbf^0$};
  \node[label,anchor=east,xshift=-2pt] at (f) {$\top^0$};
  \node[label,anchor=west,xshift=1pt] at (t) {$\bot^0$};
  \node[label,anchor=west,xshift=3pt,yshift=7pt] at (top) {$\mbt^0$};
  \node[thick,ellipse,draw,densely dotted,inner xsep=5pt,inner ysep=11pt,rotate=45] at ($0.5*(top)+0.5*(t)$) {};
  \node[thick,ellipse,draw,densely dotted,inner xsep=5pt,inner ysep=11pt,rotate=45] at ($0.5*(bot)+0.5*(f)$) {};
  \node at (-1.75,2.825) {$\delta_0$};
  \path (top) edge [curvy, bend right=15] node {} (1);
  \path (bot) edge [curvy, bend left=15] node {} (0);
\end{scope}

\begin{scope}[xshift=4cm]
  \node at (0,-1.5) {$\M_k^\flat$};
  \node[unshaded] (0) at (0,0) {};
  \node[unshaded] (f) at (0,1.4) {};
  \node[unshaded] (top) at (-1,2.4) {};
  \node[unshaded] (bot) at (1,2.4) {};
  \node[unshaded] (t) at (0,3.4) {};
  \node[unshaded] (1) at (0,4.8) {};
  \draw[order] (0) -- (f) -- (top) -- (t) -- (1);
  \draw[order] (f) -- (bot) -- (t);
  \node[label,anchor=west,xshift=1pt] at (bot) {$\bot^k$};
  \node[label,anchor=west,xshift=-2pt,yshift=-2pt] at (f) {$\mbf^k$};
  \node[label,anchor=west,xshift=5pt,yshift=5pt] at (t) {$\mbt^k$};
  \node[label,anchor=east,xshift=-9pt] at (0) {$\zero^k$};
  \node[label,anchor=east,xshift=-2pt] at (1) {$\one^k$};
  \node[label,anchor=east] at (top) {$\top^k$};
  \node[thick,ellipse,draw,densely dotted,inner xsep=4pt,inner ysep=4pt] at (1) {};
  \node[thick,ellipse,draw,densely dotted,inner xsep=14pt,inner ysep=20pt] at ($0.5*(0)+0.5*(t)$) {};
\end{scope}

\begin{scope}[xshift=8cm]
  \node at (0,-1.5) {$\boldsymbol 2$};
  \node[unshaded] (c0) at (0,1.65) {};
  \node[unshaded] (c1) at (0,3.15) {};
  \draw[order] (c0) -- (c1);
  \node at (-1.5,5) {$\gamma_k$};
  \path (1) edge [curvy,bend left=15] node {} (c1);
  \path (0) edge [curvy,bend right=15] node {} (c0);
\end{scope}

\begin{scope}[xshift=12cm]
  \node at (0,-1.5) {$\M_k^\flat$};
  \node[unshaded] (0) at (0,0) {};
  \node[unshaded] (f) at (0,1.4) {};
  \node[unshaded] (top) at (-1,2.4) {};
  \node[unshaded] (bot) at (1,2.4) {};
  \node[unshaded] (t) at (0,3.4) {};
  \node[unshaded] (1) at (0,4.8) {};
  \draw[order] (0) -- (f) -- (top) -- (t) -- (1);
  \draw[order] (f) -- (bot) -- (t);
  \node[label,anchor=west] at (bot) {$\bot^k$};
  \node[label,anchor=west,xshift=5pt,yshift=-2pt] at (f) {$\mbf^k$};
  \node[label,anchor=west,xshift=-2pt,yshift=5pt] at (t) {$\mbt^k$};
  \node[label,anchor=west,xshift=3pt] at (0) {$\zero^k$};
  \node[label,anchor=west,xshift=9pt] at (1) {$\one^k$};
  \node[label,anchor=east,xshift=-1pt] at (top) {$\top^k$};
  \node[thick,ellipse,draw,densely dotted,inner xsep=4pt,inner ysep=4pt] at (0) {};
  \node[thick,ellipse,draw,densely dotted,inner xsep=14pt,inner ysep=20pt] at ($0.5*(1)+0.5*(f)$) {};
  \node at (-2.5,5) {$\delta_k$};
  \path (1) edge [curvy,bend right=15] node {} (c1);
  \path (0) edge [curvy,bend left=15] node {} (c0);
\end{scope}
\end{tikzpicture}}
\caption{The carriers}\label{fig:carriers}
\end{figure}

The following simple lemma, which can be easily verified by the reader, is the key to the  theorem that follows it.
For all $k\in [0, n]$, define the carrier set $\Omega_k :=\{\gamma_k, \delta_k\}$.

\begin{lem}\label{lem:Pigsep}
The maps $g_1, \dots, g_n$ and the carrier sets $\Omega_0, \dots, \Omega_n$ satisfy the following separation condition:
\begin{enumerate}[label=\textup{(Sep)},leftmargin=1.8\parindent,itemsep=0pt] 
\item for all $k\in [0, n]$ and all $a, b \in M_k$ with $a \neq b$, there exists $\omega \in \Omega_k$ such that $\omega(a)\neq \omega(b)$ or there exists $\omega \in \Omega_0$ such that $\omega(g_k(a))\neq \omega(g_k(b))$.
\end{enumerate}
\end{lem}

Given that (Sep) holds, the following theorem is a direct application of Cabrer and Priestley's description of the Priestley dual via piggyback relations~\cite[Theorem~2.3]{CP-coprod}.

\begin{thm}\label{thm:HilLeo}
Let $\A\in \CV_n$. Define
\[
X = \bigcup\big\{\,X_k \mid k \in [0, n]\,\big\}\quad \text{and} \quad  Y = \bigcup\big\{\,X_k \times \Omega_k \mid k \in [0, n]\,\big\},
\]
where $X_k := \CV_n(\A, \M_k)$, for all $k\in [0, n]$.
\begin{itemize}[itemsep=0pt] 
\item Let $\Tp^Y$ be the topology on the set $Y$ induced by the natural topology on $\CV_n(\A, \M_k)$, for $k\in [0, n]$, and the discrete topology on $\Omega_k$.

\item Given $j, k\in [0,n]$, $(x, \omega_1)\in X_j \times \Omega_j$ and $(y, \omega_2)\in X_k \times \Omega_k$, define $(x, \omega_1) \leP (y, \omega_2)$ if $(x, y) \in R^X$ for some $R\in \Rsub{\omega_1 \omega_2}$, where $R^X$ is the natural pointwise extension of $R$ to $X$.
\end{itemize}
Then $\leP$ is a quasi-order on $Y$ and $\langle Y/{\approx}; {\leP}/{\approx}, \Tp/{\approx}\rangle$ is a Priestley space isomorphic to $\mathrm H(\A^\flat)$, where $\approx$ is the equivalence relation determined by $\leP$ and $\Tp/{\approx}$ is the quotient topology.
\end{thm}

To unpack this result and see how it leads directly to the proof of Theorem~\ref{thm:trans}, we need to calculate the sets $\Rsub{\omega_1 \omega_2}$ of  piggyback relations for all $\omega_1, \omega_2 \in \bigcup\{\Omega_k \mid k \in [0, n]\}$. Along the way we shall see that, in our case, the quasi-order $\leP$ defined in Theorem~\ref{thm:HilLeo} is in fact an order.

The following result from \cite{CDH19} will allow us to draw the subalgebra lattices of $\M_j\times \M_k$, for $j, k\in [0, n]$. Along with the relations $\lek$, for $k\in [0, n]$, and $\lejk$, for $j < k$,  used in the duality for $\CV_n$, we require some further relations. For $j > k$, let $\geqslant^{jk}$ (a relation from $M_j$ to $M_k$) be the converse of $\le^{kj}$, and for all $j, k\in [0, n]$, define:
\begin{align*}
S_\le^{jk} &= (g_j, g_k)^{-1}(\lez) = \{\,(a, b)\in M_j \times M_k \mid g_j(a)\lez g_k(b)\,\},\\ 
S_\ge^{jk} &= (g_j, g_k)^{-1}(\gez) = \{\,(a, b)\in M_j \times M_k \mid g_j(a)\gez g_k(b)\,\}.
\end{align*}
Note that $S_\ge^{jk}$ is the converse of $S_\le^{kj}$. In some of our calculations, we require the following explicit descriptions of $S_\le^{jk}$ and $S_\ge^{jk}$---see~\cite[Section 5.1]{CDH19}. For $k\in [0,n]$, let $\bsF_k$ and $\bsT_k$ denote, respectively, the sets of `false' constants and `true' constants in~$\M_k$. Then
\begin{equation}\tag{$\ddagger$}\label{Sjk}
\begin{aligned}
S_\le^{jk} &= \big(M_j\times \{\top^k\}\big)\cup \big(\{\bot^j\}\times M_k\big) \cup \big(\bsF_j \times \bsF_k\big) \cup \big(\bsT_j \times \bsT_k\big),\\
S_\ge^{jk} &= \big(M_j\times \{\bot^k\}\big)\cup \big(\{\top^j\}\times M_k\big) \cup \big(\bsF_j \times \bsF_k\big) \cup \big(\bsT_j \times \bsT_k\big).
\end{aligned}
\end{equation}

\begin{thm}[{\cite[Lemma 5.4, Theorem 6.1]{CDH19}}]\label{thm:MImulti}
Let $n\in \omega\comp\{0\}$.
\begin{enumerate}[label=\textup{(\arabic*)},itemsep=0pt] 

\item The meet-irreducible elements of
$\Sub(\M_0\times \M_0)$ are $\lez$ and $\gez$.

\item For all $k\in [1, n]$, the meet-irreducible elements of
$\Sub(\M_k\times \M_k)$ are $\lek$, $\gek$, $S_\le^{kk}$ and $S_\ge^{kk}$.

\item For all $k\in [1, n]$, the meet-irreducible elements of $\Sub(\M_0\times \M_k)$ are $S_\le^{0k}$ and~$S_\ge^{0k}$, and the  meet-irreducible elements of $\Sub(\M_k\times \M_0)$ are $S_\le^{k0}$ and $S_\ge^{k0}$.

\item For all $j,k\in [1, n]$, with $j \ne k$, the meet-irreducible elements of $\Sub(\M_j\times \M_k)$ are $S_\le^{jk}$ and $S_\ge^{jk}$, along with $\lejk$, if $j < k$, and $\ge^{jk}$, if $j > k$.

\end{enumerate}
\end{thm}

Throughout the remainder of this section we will use, without reference, the descriptions of $\lez$, $\lek$ and $\lejk$ given in $(\dagger)$
in Section~\ref{sec:CVduality},
and the descriptions of $S_\le^{jk}$ and $S_\ge^{jk}$ given in $(\ddagger)$.
In order to make our presentation of the arguments in this section more compact, we will often write $ab$ for the ordered pair $(a,b)$.

With Theorem~\ref{thm:MImulti} in hand, it is easy to calculate $\Sub(\M_j\times \M_k)$, for $j, k\in [0,n]$. The following observations are all that we require.
\begin{itemize}[itemsep=0pt] 

\item The pairs $\lez$ and $\gez$, $\lek$ and $\gek$, $S_\le^{kk}$ and $S_\ge^{kk}$, $S_\le^{0k}$ and $S_\ge^{0k}$, 
and $S_\le^{jk}$ and $S_\ge^{jk}$ are non-comparable.

\goodbreak 

\item Both $\lek$ and $\gek$ are subsets of $S_\le^{kk}$ and $S_\ge^{kk}$.

\item $S_\le^{jk}$ and $S_\ge^{jk}$ contain $\lejk$, if $j < k$, and contain $\ge^{jk}$, if $j > k$.

\item The bottom of $\Sub(\M_j\times \M_k)$ is the set $K_{jk}$ of constants, given by
\begin{align*}
K_{00} &= \{\top^0\top^0,\, \mbf^0\mbf^0,\, \mbt^0\mbt^0,\, \bot^0 \bot^0\} = \Delta_{M_0},\\
K_{0k} &= \{\top^0 \top^k,\,  \mbf^0\mbf^k,\, \mbf^0\zero^k,\, \mbt^0\mbt^k,\, \mbt^0\one^k, \,\bot^0\bot^k\}\\
&= \graph(g_k)\conv,
\text{ the converse of the graph of $g_k$},\\
K_{kk} &= \{\top^k\top^k,\, \mbf^k\mbf^k,\, \zero^k\zero^k,\, \mbt^k\mbt^k,\, \one^k\one^k,\, \bot^k \bot^k\} = \Delta_{M_k},\\
K_{jk} &= \{\top^j\top^k,\, \mbf^j\mbf^k,\, \mbf^j\zero^k,\,\zero^j\zero^k,\, \mbt^j\mbt^k,\, \mbt^j\one^k,\, \one^j\one^k,\, \bot^j \bot^k\} = {\lejk},
\end{align*}
where $0 <k$ in the second case and $j <k$ in the last case. In each case, the elements of $K_{jk}$ are determined by the way the constants are assigned to $\M_j$ and $\M_k$---see Definitions~\ref{def:M0} and \ref{def:Mk}.
\end{itemize}

It follows that the lattices $\Sub(\M_j\times \M_k)$ are as given in Figures~\ref{fig:subalg-I&II} and~\ref{fig:subalg-III&IV}; the meet-irreducible elements are shaded.

\begin{figure}[ht]
\centering 
\begin{tikzpicture}[scale=0.7]

\draw[order] (-5,13) -- (-6,14) -- (-5,15);
\draw[order] (-5,13) -- (-4,14) -- (-5,15);
\node[unshaded] at (-5,15){};
\node[shaded] at (-6,14) {};
\node[shaded] at (-4,14) {};
\node[unshaded] at (-5,13) {};
\node[label] at (-5,15.5){$\M_0\times \M_0$};
\node[label,anchor=east] at (-6,14){$\lez$};
\node[label,anchor=west] at (-4,14) {$\gez$};
\node[label] at (-4.91,12.5){$K_{00} = \Delta_{M_0}$};

\draw[order] (2,13) -- (1,14) -- (2,15);
\draw[order] (2,13) -- (3,14) -- (2,15);
\node[unshaded] at (2,15){};
\node[shaded] at (1,14) {};
\node[shaded] at (3,14) {};
\node[unshaded] at (2,13) {};
\node[label] at (2,15.5){$\M_0\times \M_k$};
\node[label,anchor=east] at (1,14){$S_{\le}^{0k}$};
\node[label,anchor=west] at (3,14) {$S_{\ge}^{0k}$};
\node[label,anchor=west] at (0.493,12.5){$K_{0k} = \graph(g_k)\conv$};
\end{tikzpicture}
\caption{$\Sub(\M_0\times
\M_0)$ and $\Sub(\M_0\times \M_k)$ for $0 <k$}\label{fig:subalg-I&II}
\begin{tikzpicture}[scale=0.7]

\draw[order] (-5,11) -- (-6,12) -- (-5,13) -- (-6,14) -- (-5,15);
\draw[order] (-5,11) -- (-4,12) -- (-5,13) -- (-4,14) -- (-5,15);
\node[unshaded] at (-5,15){};
\node[shaded] at (-6,14) {};
\node[shaded] at (-4,14) {};
\node[unshaded] at (-5,13) {};
\node[shaded] at (-6,12) {};
\node[shaded] at (-4,12) {};
\node[unshaded] at (-5,11) {};
\node[label] at (-5,15.5){$\M_k\times \M_k$};
\node[label,anchor=east] at (-6,14){$S_{\le}^{kk}$};
\node[label,anchor=west] at (-4,14) {$S_{\ge}^{kk}$};
\node[label,anchor=east,yshift=4pt] at (-6,12) {$\lek$};
\node[label,anchor=west,yshift=4pt] at (-4,12) {$\gek$};
\node[label] at (-4.94,10.5){$K_{kk} = \Delta_{M_k}$};

\draw[order] (2,11.5) -- (2,13) -- (1,14) -- (2,15);
\draw[order] (2,13) -- (3,14) -- (2,15);
\node[unshaded] at (2,15){};
\node[shaded] at (1,14) {};
\node[shaded] at (3,14) {};
\node[unshaded] at (2,13) {};
\node[shaded] at (2,11.5) {};
\node[label] at (2,15.5){$\M_j\times \M_k$};
\node[label,anchor=east] at (1,14){$S_{\le}^{jk}$};
\node[label,anchor=west] at (3,14) {$S_{\ge}^{jk}$};
\node[label] at (1.97,11){$K_{jk} = {\lejk}$};
\end{tikzpicture}
\caption{$\Sub(\M_k\times
\M_k)$ and $\Sub(\M_j\times
\M_k)$ for $j <k$}\label{fig:subalg-III&IV}
\end{figure}

\subsection{The piggyback relations}
The following simple lemma will more or less cut our work in half. Define $\Omega = \bigcup\{\, \Omega_k \mid k\in [0, n]\,\}$, and define $' \colon \Omega \to \Omega$ by $\gamma_k' = \delta_k$ and $\delta_k' = \gamma_k$, for all $k\in [0, n]$. For a subset $S$ of an algebra $\A\in \CV_n$, define $\neg S =\{\,\neg a \mid a \in S\,\}$. Given a set $\mathcal R$ of multi-sorted binary relations, define $\mathcal R\conv =\{\, R\conv \mid R\in \mathcal R\,\}$,
where $R\conv$ is the converse of~$R$.

\begin{lem}\label{lem:sym} For all $\omega,\omega_1, \omega_2\in \Omega$, we have
\begin{enumerate}[label=\textup{(\alph*)},itemsep=0pt] 

\item $\omega' = c\circ \omega \circ \neg$, where $c\colon\{0, 1\} \to \{0, 1\}$ is complementation.

\item $(\omega_2, \omega_1)^{-1}(\le) = \neg\big((\omega_1', \omega_2')^{-1}(\le)\conv\,\big)$,

\item $\Rsub{\omega_2\omega_1} = \mathcal R\convsub{\omega_1'\omega_2'}$.

\end{enumerate}
\end{lem}

\begin{proof}
(a) This is a simple calculation.

(b) Using the fact that $\neg$ is an involution, and that $\neg$ and $\conv$ commute, we have
\begin{align*}
(\omega_2, \omega_1)^{-1}(\le) &= \{\, (a,b) \mid \omega_2(a) \le \omega_1(b)\,\}\\
&= \{\, (b,a) \mid \omega_1(b) \ge \omega_2(a)\,\}\conv\\
&= \{\, (b,a) \mid c(\omega_1(b)) \le c(\omega_2(a))\,\}\conv\\
&= \{\, (\neg u,\neg v) \mid c(\omega_1(\neg u)) \le c(\omega_2(\neg v))\,\}\conv\\
&= \{\, (\neg u,\neg v) \mid \omega_1'(u) \le \omega_2'(v)\,\}\conv\\
&= \big(\neg(\omega_1', \omega_2')^{-1}(\le)\big)\conv\\
&= \neg\big((\omega_1', \omega_2')^{-1}(\le)\conv\,\big).
\end{align*}

(c) This follows at once from (b) as subuniverses are closed under~$\neg$.
\end{proof}

In searching for maximal subalgebras in $(\omega_1, \omega_2)^{-1}(\leqslant)$ in the products $\M_j\times \M_k$, $j, k\in [0,n]$, we consider
four cases.

\smallskip

\noindent I. The case $\M_0\times \M_0$. The set
\begin{align*}
(\delta_0, \gamma_0)^{-1}(\leqslant) &= \big(\{\mbf^0, \top^0\} \times M_0\big) \cup \big(M_0\times \{\top^0, \mbt^0\}\big),
\end{align*}
does not contain $K_{00} = \Delta_{M_0}$, as it does not contain $\bot^0 \bot^0$,
whence it contains no subalgebra of $\M_0\times \M_0$.  Similarly,
\[
(\gamma_0, \delta_0)^{-1}(\leqslant) = \big(\{\mbf^0, \bot^0\} \times M_0\big) \cup \big(M_0\times \{\bot^0, \mbt^0\}\big),
\]
does not contain $K_{00} = \Delta_{M_0}$, as it does not contain $\top^0 \top^0$.
Hence $\Rsub{\delta_0\gamma_0} = \Rsub{\gamma_0\delta_0}  =\varnothing$.
The set
\[
(\gamma_0, \gamma_0)^{-1}(\leqslant) = \big(\{\mbf^0, \bot^0\} \times M_0\big) \cup \big(M_0\times \{\top^0, \mbt^0\}\big)
\]
contains ${\lez}$ but not ${\gez}$ as $\mbt^0\bot^0\notin (\gamma_0, \gamma_0)^{-1}(\leqslant)$.
It follows from the shape of $\Sub(\M_0 \times \M_0)$ (see Figure~\ref{fig:subalg-I&II}) that $\Rsub{\gamma_0\gamma_0} = \{\lez\}$, and hence
\[
\Rsub{\delta_0\delta_0} = \mathcal R\convsub{\delta_0'\delta_0'} = \mathcal R\convsub{\gamma_0 \gamma_0}= \{\lez\}\conv = \{\gez\}
\]
by Lemma~\ref{lem:sym}(c).
\smallskip

\noindent II. The case $\M_0\times \M_k$ and $\M_k\times \M_0$, for $k\in [1, n]$. Since
\[
\bot^0\bot^k\notin (\delta_0, \gamma_k)^{-1}(\leqslant) \quad\text{and}\quad \top^0\top^k\notin (\gamma_0, \gamma_k)^{-1}(\leqslant),
\]
these sets do not contain  $K_{0k} = \graph(g_k)\conv$, whence $\Rsub{\delta_0\gamma_k} = \Rsub{\gamma_0\gamma_k} =\varnothing$. The set
\[
(\gamma_0, \delta_k)^{-1}(\leqslant) = \big(\{\mbf^0, \bot^0\} \times M_k\big) \cup \big(M_0\times(M_k\comp \{\zero\})\big),
\]
contains $S^{0k}_\le$, but not $S^{0k}_\ge$ as $\top^0\zero^k \notin (\gamma_0, \delta_k)^{-1}(\leqslant)$. Similarly, the set
\[
(\delta_0, \delta_k)^{-1}(\leqslant) = \big(\{\mbf^0, \top^0\} \times M_k\big) \cup \big(M_0\times (M_k\comp \{\zero^k\})\big),
\]
contains $S^{0k}_\ge$, but not $S^{0k}_\le$ as $\bot^0\zero^k \notin (\gamma_0, \delta_k)^{-1}(\leqslant)$. Hence the shape of $\Sub(\M_0 \times \M_k)$ (see Figure~\ref{fig:subalg-I&II}) tells us that $\Rsub{\gamma_0\delta_k} = \{S^{0k}_\le\}$ and $\Rsub{\delta_0\delta_k} = \{S^{0k}_\ge\}$.

By Lemma~\ref{lem:sym}(c) we conclude that
\begin{align*}
\Rsub{\gamma_k\delta_0} &= \mathcal R\convsub{\delta_0'\gamma_k'} = \mathcal R\convsub{\gamma_0\delta_k} = \{S^{0k}_\le\}\conv = \{S^{k0}_\ge\}, \\
\Rsub{\gamma_k\gamma_0} &= \mathcal R\convsub{\gamma_0'\gamma_k'} = \mathcal R\convsub{\delta_0\delta_k} = \{S^{0k}_\ge\}\conv = \{S^{k0}_\le\},\\
\Rsub{\delta_k\gamma_0} &= \mathcal R\convsub{\gamma_0'\delta_k'} = \mathcal R\convsub{\delta_0\gamma_k} = \varnothing,
\text{ and}\\
\Rsub{\delta_k\delta_0} &= \mathcal R\convsub{\delta_0'\delta_k'} = \mathcal R\convsub{\gamma_0\gamma_k} = \varnothing.
\end{align*}

\smallskip

\noindent III. The case $\M_k\times \M_k$, for $k\in [1, n]$. The set $(\delta_k, \gamma_k)^{-1}(\leqslant)$ does not contain $K_{kk} = \Delta_{M_k}$ as it does not contain $\top^k\top^k$, for example. Hence $\Rsub{\delta_k\gamma_k} =\varnothing$.
The set
\[
(\gamma_k, \gamma_k)^{-1}(\leqslant) = \big(M_k\comp \{\one^k\} \times M_k\big) \cup \big(M_k\times \{\one^k\}\big)
\]
contain $\lek$ but not $\gek$ as it does not contain $\one^k\mbt^k$. Hence the shape of $\Sub(\M_k \times \M_k)$ (see Figure~\ref{fig:subalg-III&IV}) tells us that $\Rsub{\gamma_k\gamma_k} = \{\lek\}$. By Lemma~\ref{lem:sym}(c) it follows that
\[
\Rsub{\delta_k\delta_k} = \mathcal R\convsub{\delta_k'\delta_k'}  = \mathcal R\convsub{\gamma_k\gamma_k} = \{\lek\}\conv =\{\gek\}.
\]
Finally, the set
\[
(\gamma_k, \delta_k)^{-1}(\leqslant) = \big((M_k\comp \{\one^k\}) \times M_k\big) \cup \big(M_k\times (M_k\comp \{\zero^k\})\big)
\]
contains both $S^{kk}_\le$ and $S^{kk}_\ge$, and again the shape of the lattice
$\Sub(\M_k \times \M_k)$ (see Figure~\ref{fig:subalg-III&IV}) tells us that $\Rsub{\gamma_k\delta_k} = \{S^{kk}_\le, S^{kk}_\ge\}$.

\smallskip

\noindent IV. The case $\M_j\times \M_k$ and $\M_k\times \M_j$, for $0 < j < k$. Let $j, k\in [1, n]$ with $j < k$. The set $(\delta_j, \gamma_k)^{-1}(\leqslant)$ does not contain $K_{jk}= {\lejk}$ as it does not contain $\top^j\top^k$, for example. Similarly, the set $(\delta_j, \delta_k)^{-1}(\leqslant)$
does not contain $K_{jk}= {\lejk}$ as it does not contain $\mbf^j\zero^k$. Consequently, $\Rsub{\delta_j\gamma_k} =\Rsub{\delta_j\delta_k} =\varnothing$.
The set
\[
(\gamma_j, \delta_k)^{-1}(\leqslant) = \big((M_j\comp \{\one^j\}) \times M_k\big) \cup \big(M_j\times (M_k\comp \{\zero^k\})\big)
\]
contains both $S^{jk}_\le$ and $S^{jk}_\ge$, and the shape of 
$\Sub(\M_j \times \M_k)$ (see Figure~\ref{fig:subalg-III&IV}) tells us that $\Rsub{\gamma_j\delta_k} = \{S^{jk}_\le, S^{jk}_\ge\}$.
Finally, the set
\[
(\gamma_j, \gamma_k)^{-1}(\leqslant) = \big((M_j\comp \{\one\}) \times M_k\big) \cup \big(M_j\times \{\one\}\big)
\]
contains $\lejk$ but does not contain $S^{jk}_\le\cap S^{jk}_\ge$ as $\one^j\mbt^k \in S^{jk}_\le\cap S^{jk}_\ge \comp (\gamma_j, \gamma_k)^{-1}(\leqslant)$. Hence the shape of $\Sub(\M_j \times \M_k)$ (see Figure~\ref{fig:subalg-III&IV}) tells us that $\Rsub{\gamma_j\gamma_k} = \{\lejk\}$.

Applying Lemma~\ref{lem:sym}(c) once again, we conclude that
\begin{align*}
\Rsub{\gamma_k\delta_j} &= \mathcal R\convsub{\delta_j'\gamma_k'} = \mathcal R\convsub{\gamma_j\delta_k} = \{S^{jk}_\le, S^{jk}_\ge\}\conv = \{S^{kj}_\ge, S^{kj}_\le\},\\
\Rsub{\delta_k\delta_j} &= \mathcal R\convsub{\delta_j'\delta_k'} = \mathcal R\convsub{\gamma_j\gamma_k} = \{\lejk\}\conv = \{\gekj\},\\
\Rsub{\gamma_k\gamma_j} &= \mathcal R\convsub{\gamma_j'\gamma_k'} = \mathcal R\convsub{\delta_j\delta_k} = \varnothing, \text{ and}\\
\Rsub{\delta_k\gamma_j} &= \mathcal R\convsub{\gamma_j'\delta_k'} = \mathcal R\convsub{\delta_j\gamma_k} = \varnothing.
\end{align*}

\smallskip

The sets $\Rsub{\omega_1\omega_2}$ of piggyback relations, for $\omega_1, \omega_2\in \Omega$ are presented in the Table~\ref{table:multpig};
to save space the braces have been removed from the non-empty sets. The numbers in brackets will be used in the next subsection.

\begin{table}[ht]
 \caption{The sets $\mathcal R_{\omega_1 \omega_2}$ of piggyback relations}
 \label{table:multpig}

\centering
\begin{tabular}{|c|c||ccc|ccc|cc|}
\hline
\multicolumn{10}{|c|}{Piggyback relations for $0 < k$ and $0 < j < k$} \\
\hline
\rule[-6pt]{0pt}{10pt}Alg & & \multicolumn{3}{|c|}{$\M_0$} & \multicolumn{3}{|c|}{$\M_j$} & \multicolumn{2}{|c|}{$\M_k$}\\
\cline{2-10}
\rule[-0pt]{0pt}{10pt} &  $\omega_i$ & $\gamma_0$ & & $\delta_0$ & & $\gamma_j$  & $\delta_j$ & $\gamma_k$ & $\delta_k$\\
\hline
\hline
\rule[-6pt]{0pt}{16pt}$\M_0 $& $\gamma_0$ & $\lez${\tiny (1)}&  & $\varnothing$ & & & & $\varnothing$ & $S^{0k}_\le${\tiny (5)}\\
\rule[-6pt]{0pt}{0pt}$ $& $\delta_0$ & $\varnothing$ &  & $\gez ${\tiny (2)} & & & & $\varnothing$ & $S^{0k}_\ge${\tiny (6)}\\
\hline
\rule[-6pt]{0pt}{16pt}$\M_j $& $\gamma_j$ &  &  &  & & & & $\lejk ${\tiny (1)} & $S^{jk}_\le, S^{jk}_\ge ${\tiny (7)}\\
\rule[-6pt]{0pt}{0pt}$ $& $\delta_j$ &  &  &  & & & & $\varnothing$ & $\varnothing$\\
\hline
\rule[-6pt]{0pt}{16pt}$\M_k $& $\gamma_k$ & $S^{k0}_\le${\tiny (3)} &  & $S^{k0}_\ge${\tiny (4)}  & & $\varnothing$  &$S^{kj}_\ge,
S^{kj}_\le${\tiny (7)}  & $\lek${\tiny (1)} & $S^{kk}_\le, S^{kk}_\ge ${\tiny (7)}\\
\rule[-6pt]{0pt}{0pt}$ $& $\delta_k$ & $\varnothing$ &  & $\varnothing$ &  & $\varnothing$   & $\gekj${\tiny (2)} & $\varnothing$ & $\gek${\tiny (2)}\\
\hline
\hline
 \end{tabular}
\end{table}

\subsection{The proof of Theorem~\ref{thm:trans}}

We apply Theorem~\ref{thm:HilLeo}. Let $\A\in \CV_n$ and let $\A^\flat = \langle A; \wedge, \vee, \mbf_0, \mbt_0\rangle$ be its bounded-distributive-lattice
reduct. Define
\[
X = \bigcup\big\{\,X_k \mid k \in [0, n]\,\big\}\quad \text{and} \quad  Y = \bigcup\big\{\,X_k \times \Omega_k \mid k \in [0, n]\,\big\},
\]
where $X_k := \CV_n(\A, \M_k)$, for all $k\in [0, n]$. Endow $Y$ with the quasi-order $\leP$ and topology as described in Theorem~\ref{thm:HilLeo}.

\begin{lem}\label{lePorder}
The quasi-order $\leP$ on $Y$ is an order.
\end{lem}
\begin{proof}
Let $(x, \omega_1), (y, \omega_2)\in Y$ with $(x, \omega_1) \leP (y, \omega_2)$ and $(y, \omega_2) \leP (x, \omega_1)$. Then $x\in X_j$, $y\in X_k$, $\omega_1\in \{\gamma_j, \delta_j\}$ and  $\omega_2\in \{\gamma_k, \delta_k\}$, for some $j, k\in [0, n]$, and by the definition of $\leP$ on $Y$, we have
\[
(x, y) \in R^X \text{ for some } R\in \Rsub{\omega_1 \omega_2}, \text{ and } (y, x) \in R^X \text{ for some } R\in \Rsub{\omega_2 \omega_1}.
\]

There are many cases to consider, but each corresponds to either a diagonal entry in Table~\ref{table:multpig} or to a symmetric pair of off-diagonal entries in the table. On the diagonal, where $\omega_1 = \omega_2$, the relation
 $R$ is one of $\lez$, $\gez$, $\lek$ and $\gek$, each of which is an order, so $(x, y), (y,x)\in R^X$ immediately gives $(x, \omega_1) = (y, \omega_2)$.
For each symmetric pair of off-diagonal entries in the table we have either $\Rsub{\omega_1 \omega_2} = \varnothing$ or $\Rsub{\omega_2 \omega_1} = \varnothing$, so these cases do not actually occur.
\end{proof}

Since $\leP$ is an order, Theorem~\ref{thm:HilLeo} now tells us that $\langle Y; {\leP}, \Tp\rangle$ is a Priestley space isomorphic to $\mathrm H(\A^\flat)$. Let
\[
\X = \mathrm D(\A) = \langle X_0 \du \cdots \du X_n  ; \mathcal G_{(n)}, \mathcal {S}_{(n)}, \Tp\rangle
\]
be the dual of~$\A$. Recall from Definition~\ref{def:P(X)} that
\[
\mathrm P(\X) = \langle X\du \widehat X; \leP, \Tp\rangle,
\]
where $\widehat X := \{\widehat x \mid x \in X\}$. Define $\eta\colon \mathrm P(\X) \to Y$ by
\[
\eta(x) = (x,\gamma_k) \text{ and } \eta(\widehat x) = (x,\delta_k)
\]
for all $x\in X_k$ and all $k\in [0, n]$. It is clear that $\eta$ is a homeomorphism between the underlying topological spaces of $\mathrm P(\X)$ and of the Priestley space $\langle Y; \leP, \Tp\rangle$. To complete the proof of Theorem~\ref{thm:trans} it remains to prove that $x_1 \leP x_2$ in $\mathrm P(\X)$ if and only if $\eta(x_1) \leP \eta(x_2)$ in $Y$. Once again, this can be read straight off Table~\ref{table:multpig}: the cells of the table corresponding to Condition~$(n)$ in Definition~\ref{def:P(X)} have been labelled~$(n)$. For example, using Part $(5)$ of the definition of $\leP$ on $\mathrm P(\X)$ and the cell labelled~$(5)$ in the table, along with the fact that
\[
(x, y) \in S^{0k}_\le \quad \iff \quad x = g_0(x) \lez g_k(y),
\]
we have, for all $x\in X_0$ and $\widehat y \in \widehat X\comp \widehat X_0$,
\begin{alignat*}{2}
x \leP \widehat y &\iff x\le g(y)&\qquad&\text{definition of $\leP$ on $\mathrm P(\X)$}\\
                  &\iff x \lez g_k(y)&&\\
                  &\iff (x, y) \in S^{0k}_\le&&\\
                  &\iff (x, \gamma_0)\leP (y, \delta_k)&\qquad &\text{definition of $\leP$ on $Y$}\\
                  &\iff \eta(x) \leP \eta(\widehat y).
\end{alignat*}
This completes the proof of Theorem~\ref{thm:trans}.\hfill\qed

\smallskip

Some observations on this proof lead us to close this paper by posing two questions.

Each of our piggyback relations is meet-irreducible in the appropriate subalgebra lattice $\Sub(\M_j \times \M_k)$. This is not always so. For example, the piggyback relations on the three-element Kleene algebra~$\K$ are calculated on page 155 of Davey and Priestley~\cite{DP-multi}; in Clark and Davey~\cite{CD98} they are denoted by $\leP$, $\geP$ and $\sim$. Each of these relations is meet-irreducible (see the diagram of $\Sub(\K^2)$ on page 247 of~\cite{CD98}). The remaining piggyback relation, $\Delta_{K_0} = \{ 00, 11\}$, is the bottom of $\Sub(\K^2)$ and not meet-irreducible.

\begin{prob}
Investigate the relationship between meet-irreducibles and piggyback relations
in the appropriate subalgebra lattice. In particular, give sufficient conditions under which all (multi-sorted) piggyback relations are meet-irreducible.
\end{prob}

Our Theorem~\ref{thm:HilLeo} is a special case of Cabrer and Priestley's description of the Priestley dual via piggyback relations~\cite[Theorem~2.3]{CP-coprod}. While their theorem guarantees only that $\leP$ is a quasi-order, in our case it was antisymmetric and hence was actually an order. For this to occur we would certainly need the set $\CM = \{\M_0, \dots, \M_n\}$ of algebras to be a minimal $\ISP$-generating set for the quasi-variety $\ISP(\CM)$. We would also make the set $\mathcal G$ of multi-sorted endomorphisms as large as possible and the set $\Omega$ of carriers as small as possible while guaranteeing that the separation condition $(S)$ of the Multi-sorted Piggyback Duality Theorem~\cite[Theorem 7.2.1]{CD98} is satisfied. But what else is required?

\begin{prob}
Give sufficient conditions for the quasi-order $\leP$ defined in~\cite[Theorem~2.3]{CP-coprod} to be an order.
\end{prob}

\section*{Acknowledgements}

The first author is grateful to La Trobe University for his appointment as an Honorary Visiting Fellow in the School of Engineering and Mathematical Sciences for August 2019. 
The third author acknowledges his appointment as a Visiting Professor at the University of Johannesburg from 1 June 2020. 
The first and third author would like to thank La Trobe University for its hospitality during a visit there in August 2019.


\end{document}